\documentclass[12pt, article]{amsart}
\usepackage{amsmath,amsfonts,amssymb,amsthm}
\usepackage[abbrev,nobysame]{amsrefs}
\usepackage{mathrsfs}
\usepackage[all,cmtip]{xy}
\usepackage{tikz-cd}
 \usepackage{relsize}
\usepackage{hyperref}
\usepackage{accents}
\usepackage{xcolor}
\usepackage{soul}\usepackage{calligra}
\usepackage{enumitem}
\usepackage{bbm}

\setlist[itemize]{align=parleft,left=0pt..1.5em}

 \allowdisplaybreaks

\DeclareMathAlphabet{\mathcalligra}{T1}{calligra}{m}{n}

\DeclareMathOperator{\Res}{Res}

\DeclareMathOperator{\Hom}{Hom}

\DeclareMathOperator{\Span}{Span}

\DeclareMathOperator{\End}{End}

\DeclareMathOperator{\obj}{obj}

\begin{document}

%--------------<Theorem Style Head>--------------
\newtheorem{thm}{Theorem}[section]
\newtheorem{prop}[thm]{Proposition}
\newtheorem{coro}[thm]{Corollary}
\newtheorem{conj}[thm]{Conjecture}
\newtheorem{example}[thm]{Example}
\newtheorem{lem}[thm]{Lemma}
\newtheorem{rem}[thm]{Remark}
\newtheorem{hy}[thm]{Hypothesis}
\newtheorem*{acks}{Acknowledgements}
\theoremstyle{definition}
\newtheorem{de}[thm]{Definition}
\newtheorem{ex}[thm]{Example}

\newtheorem{convention}[thm]{Convention}

\newtheorem{bfproof}[thm]{{\bf Proof}}
%\xymatrixcolsep{5pc}
%--------------<\Theorem Style Head>-------------

%--------------<Common Sets>---------------------
\newcommand{\C}{{\mathbb{C}}}
\newcommand{\Z}{{\mathbb{Z}}}
\newcommand{\N}{{\mathbb{N}}}
\newcommand{\te}[1]{\mbox{#1}}
\newcommand{\set}[2]{{
    \left.\left\{
        {#1}
    \,\right|\,
        {#2}
    \right\}
}}
\newcommand{\sett}[2]{{
    \left\{
        {#1}
    \,\left|\,
        {#2}
    \right\}\right.
}}

\newcommand{\choice}[2]{{
\left[
\begin{array}{c}
{#1}\\{#2}
\end{array}
\right]
}}
\def \<{{\langle}}
\def \>{{\rangle}}

\def\({\left(}

\def\){\right)}

\def \:{\mathopen{\overset{\circ}{
    \mathsmaller{\mathsmaller{\circ}}}
    }}
\def \;{\mathclose{\overset{\circ}{\mathsmaller{\mathsmaller{\circ}}}}}

\newcommand{\overit}[2]{{
    \mathop{{#1}}\limits^{{#2}}
}}
\newcommand{\belowit}[2]{{
    \mathop{{#1}}\limits_{{#2}}
}}

\newcommand{\wt}[1]{\widetilde{#1}}

\newcommand{\wh}[1]{\widehat{#1}}

\newcommand{\wck}[1]{\reallywidecheck{#1}}

\newlength{\dhatheight}
\newcommand{\dwidehat}[1]{%
    \settoheight{\dhatheight}{\ensuremath{\widehat{#1}}}%
    \addtolength{\dhatheight}{-0.45ex}%
    \widehat{\vphantom{\rule{1pt}{\dhatheight}}%
    \smash{\widehat{#1}}}}
\newcommand{\dhat}[1]{%
    \settoheight{\dhatheight}{\ensuremath{\hat{#1}}}%
    \addtolength{\dhatheight}{-0.35ex}%
    \hat{\vphantom{\rule{1pt}{\dhatheight}}%
    \smash{\hat{#1}}}}

\newcommand{\dwh}[1]{\dwidehat{#1}}

\newcommand{\dis}{\displaystyle}

\newcommand{\pd}[1]{\frac{\partial}{\partial {#1}}}

\newcommand{\pdiff}[2]{\frac{\partial^{#2}}{\partial #1^{#2}}}

%--------------<\Common Sets>--------------------

%--------------<Global>--------------------------
\newcommand{\g}{{\mathfrak g}}
\newcommand{\ff}{{\mathfrak f}}
\newcommand{\f}{\ff}
\newcommand{\gc}{{\bar{\g'}}}
\newcommand{\h}{{\mathfrak h}}
\newcommand{\cent}{{\mathfrak c}}
\newcommand{\notc}{{\not c}}
\newcommand{\Loop}{{\mathcal L}}
\newcommand{\G}{{\mathcal G}}
\newcommand{\D}{\mathcal D}
\newcommand{\T}{\mathcal T}
\newcommand{\Free}{\mathcal F}
\newcommand{\Cfk}{\mathcal C}
\newcommand{\nil}{\mathfrak n}
\newcommand{\al}{\alpha}
\newcommand{\be}{\beta}
\newcommand{\beck}{\be^\vee}
\newcommand{\ssl}{{\mathfrak{sl}}}
\newcommand{\id}{\te{id}}
\newcommand{\rtu}{{\xi}}
\newcommand{\period}{{N}}
\newcommand{\half}{{\frac{1}{2}}}
\newcommand{\reciprocal}[1]{{\frac{1}{#1}}}
\newcommand{\inverse}{^{-1}}
\newcommand{\inv}{\inverse}
\newcommand{\SumInZm}[2]{\sum\limits_{{#1}\in\Z_{#2}}}
\newcommand{\uce}{{\mathfrak{uce}}}
\newcommand{\Rcat}{\mathcal R}
\newcommand{\cS}{{\mathcal{S}}}

%--------------<\Global>-------------------------

%--------------<Local>---------------------------
\newcommand{\E}{{\mathcal{E}}}
\newcommand{\F}{{\mathcal{F}}}

\newcommand{\Etopo}{{\mathcal{E}_{\te{topo}}}}

\newcommand{\Ye}{{\mathcal{Y}_\E}}

\newcommand{\rh}{{{\bf h}}}
\newcommand{\rp}{{{\bf p}}}
\newcommand{\rrho}{{{\pmb \varrho}}}
\newcommand{\ral}{{{\pmb \al}}}

%% ----------------<functors>--------------------
\newcommand{\comp}{{\mathfrak{comp}}}
\newcommand{\ctimes}{{\widehat{\boxtimes}}}
\newcommand{\ptimes}{{\widehat{\otimes}}}
\newcommand{\ptimeslt}{{
%   \leftidx{ _{\te{tri}}}{\ctimes}{}
{}_{\te{t}}\ptimes
}}
\newcommand{\ptimesrt}{{\ot_{\te{t}} }}
\newcommand{\ttp}[1]{{
    {}_{{#1}}\ptimes
}}
\newcommand{\bigptimes}{{\widehat{\bigotimes}}}
\newcommand{\bigptimeslt}{{
%   \leftidx{ _{\te{tri}}}{\ctimes}{}
{}_{\te{t}}\bigptimes
}}
\newcommand{\bigptimesrt}{{\bigptimes_{\te{t}} }}
\newcommand{\bigttp}[1]{{
    {}_{{#1}}\bigptimes
}}

\newcommand{\ot}{\otimes}
\newcommand{\Ot}{\bigotimes}
\newcommand{\bt}{\boxtimes}

\newcommand{\affva}[1]{V_{\wh\g}\(#1,0\)}
\newcommand{\saffva}[1]{L_{\wh\g}\(#1,0\)}
\newcommand{\saffmod}[1]{L_{\wh\g}\(#1\)}

\newcommand{\otcopies}[2]{\belowit{\underbrace{{#1}\ot \cdots \ot {#1}}}{{#2}\te{-times}}}

\newcommand{\wtotcopies}[3]{\belowit{\underbrace{{#1}\wh\ot_{#2} \cdots \wh\ot_{#2} {#1}}}{{#3}\te{-times}}}

%% ----------------<\functors>-------------------

%% ----------------<algebras>--------------------
\newcommand{\tar}{{\mathcal{DY}}_0\(\mathfrak{gl}_{\ell+1}\)}
\newcommand{\U}{{\mathcal{U}}}
\newcommand{\htar}{\mathcal{DY}_\hbar\(A\)}
\newcommand{\hhtar}{\widetilde{\mathcal{DY}}_\hbar\(A\)}
\newcommand{\htarz}{\mathcal{DY}_0\(\mathfrak{gl}_{\ell+1}\)}
\newcommand{\hhtarz}{\widetilde{\mathcal{DY}}_0\(A\)}
\newcommand{\qhei}{\U_\hbar\left(\hat{\h}\right)}
\newcommand{\n}{{\mathfrak{n}}}
\newcommand{\vac}{{{\mathbbm 1}}}
\newcommand{\vtar}{{{
    \mathcal{V}_{\hbar,\tau}\left(\ell,0\right)
}}}

\newcommand{\qtar}{
    \U_q\(\wh\g_\mu\)}
\newcommand{\rk}{{\bf k}}
% ----------------<\algebras>-------------------

\newcommand{\hctvs}[1]{Hausdorff complete linear topological vector space}
\newcommand{\hcta}[1]{Hausdorff complete linear topological algebra}
\newcommand{\ons}[1]{open neighborhood system}
\newcommand{\B}{\mathcal{B}}
\newcommand{\rx}{{\bf x}}
\newcommand{\re}{{\bf e}}
\newcommand{\rphi}{{\boldsymbol{ \phi}}}

\newcommand{\der}{\mathcal D}

\newcommand{\prodlim}{\mathop{\prod_{\longrightarrow}}\limits}

\newcommand{\lp}[1]{\mathcal L(f)}

\newcommand{\cha}{\check a}
\newcommand{\chh}{\check \h}

%--------------<\Local>--------------------------

\makeatletter
\@addtoreset{equation}{section}
\def\theequation{\thesection.\arabic{equation}}
\makeatother \makeatletter

\title{Representations of quantum lattice vertex algebras}

\author{Fei Kong$^1$}
\email{kongmath@hunnu.edu.cn}
\address{Key Laboratory of Computing and Stochastic Mathematics (Ministry of Education), School of Mathematics and Statistics, Hunan Normal University, Changsha, China 410081}
\thanks{$^1$Partially supported by the NSF of China  (No.12371027).}

%\begin{titlepage}
%\title{Quantization of parafermion vertex algebras}
%\author{Fei Kong}
%\newcommand\institute{%
%Key Laboratory of Computing and Stochastic Mathematics (Ministry of Education), School of Mathematics and Statistics, Hunan Normal University, Changsha, China 410081}
%\makeatletter
%\centering
%{\Huge\bfseries\sffamily\@title} \bigskip\par
%{\Large\bfseries\@author} \bigskip\par
%\email{kongmath@hunnu.edu.cn}
%\makeatother
%\vfill
%\large\institute
%\end{titlepage}

%------
% Insert an abstract.
%------
\begin{abstract}
Let $Q$ be a non-degenerate even lattice, let $V_Q$ be the lattice vertex algebra associated to $Q$, and let $V_Q^\eta$ be a quantum lattice vertex algebra (\cite{JKLT-Quantum-lattice-va}).
In this paper, we prove that every $V_Q^\eta$-module is completely reducible, and the set of simple $V_Q^\eta$-modules are in one-to-one correspondence with the set of cosets of $Q$ in its dual lattice.
\end{abstract}

%------
% Optional: Dedication.
%------
%\dedication{This memoir is dedicated to XXX.}

%------
% Insert a list of keywords.
% -- Separate keywords with comma.
% -- Capitalize only the first keyword in the list.
% -- No final full stop.
%------

%------
% Insert MSC 2020 codes according to www.ams.org/msc/msc2020.html.
% -- There must be exactly *one* primary code.
% -- The number of secondary codes is not specified.
% -- No final full stop.
%------
%------
% Insert acknowledgments.
%------
%\begin{ack}
%Part of this paper was finished during my visit at Xiamen University and Tianyuan Mathematical Center in Southeast China, in
%August 2023. I am very grateful to Professor Shaobin Tan, Fulin Chen, Qing Wang for their hospitality.
%\end{ack}
%
%%------
%% Insert information regarding funding.
%%------
%\begin{funding}
%NSF of China (No.12371027).
%\end{funding}
\maketitle

%\tableofcontents

%------
% Insert the body of the book here.
%------
\section{Introduction}

Let $Q$ be a non-degenerate even lattice of finite rank, and let $V_Q$ be the corresponding lattice vertex algebra (see \cite{Bor, FLM2}).
It is well known that every $V_Q$-module is completely reducible and the set of simple $V_Q$-modules are in one-to-one correspondence with the set of cosets of $Q$ in its dual lattice $Q^0$ (see \cite{D1,DLM, LL}).

In the general field of vertex algebras, one conceptual problem is to develop suitable quantum vertex algebra theories and associate quantum vertex algebras to quantum affine algebras.
In \cite{EK-qva}, Etingof and Kazhdan developed a theory of quantum vertex operator algebras in the sense of formal deformations of vertex algebras. Partly motivated by their work, H. Li
conducted a series of studies.
While vertex algebras are analogues of commutative associative algebras, H. Li introduced the notion of
nonlocal vertex algebras \cite{Li-nonlocal}, which are analogues of noncommutative associative algebras.
A nonlocal vertex algebra satisfying the $S$-locality becomes a weak quantum vertex algebra \cite{Li-nonlocal}.
In addition, a quantum vertex algebra \cite{Li-nonlocal} is a weak quantum vertex algebra such that the $S$-locality is controlled by a rational quantum Yang-Baxter operator.
In order to deal with quantum affine algebras, H. Li developed the theory of $\phi$-coordinated modules in \cite{Li-phi-coor}.
The $\hbar$-adic counterparts of these notions were introduced in \cite{Li-h-adic}.
In this framework, a quantum vertex operator algebra in the sense of Etingof-Kazhdan is an $\hbar$-adic quantum vertex algebra whose classical limit is a vertex algebra.

In the very paper \cite{EK-qva}, Etingof and Kazhdan constructed quantum vertex operator algebras as formal deformations of universal affine vertex algebras of type $A$, by using the $R$-matrix type relations given in \cite{RS-RTT}.
Butorac, Jing and Ko\v{z}i\'{c} (\cite{BJK-qva-BCD}) extended Etingof-Kazhdan's construction
to type $B$, $C$ and $D$ rational $R$-matrices.
For trigonometric $R$-matrices of type $A$, $B$, $C$ and $D$, Ko\v{z}i\'{c} constructed the corresponding quantum vertex operator algebras in \cite{Kozic-qva-tri-A, K-qva-phi-mod-BCD},
and established a one-to-one correspondence between their $\phi$-coordinated modules and
restricted modules for quantum affine algebras of classical types.
In our joint work with N. Jing, H. Li, and S. Tan \cite{JKLT-Defom-va}, we introduced a method for constructing quantum vertex operator algebras by using vertex bialgebras.
Utilizing this approach, we constructed a family of quantum lattice vertex algebras, which are formal deformations of lattice vertex algebras.
Moreover, based on the Drinfeld's quantum affinization construction (\cite{Dr-new,J-KM,Naka-quiver}), we constructed the quantum affine vertex algebras for all symmetric Kac-Moody Lie algebras in \cite{K-Quantum-aff-va},
and prove a one-to-one correspondence between their $\phi$-coordinated modules and certain restricted modules for quantum affinizations of symmetrizable quantum Kac-Moody algebras.

The main purpose of this paper is to study the modules of quantum lattice vertex algebras.
We demonstrate that every $ V_Q^\eta $-module is completely reducible, and the set of simple $V_Q^\eta$-modules are in
one-to-one correspondence with $Q^0/Q$.

The paper is organized as follows.
Section \ref{sec:VAs} recalls the basics about lattice vertex algebras.
Section \ref{sec:qvas} introduces the notion $\hbar$-adic quantum vertex algebras and present some basic results.
Section \ref{sec:qlattice} recalls the construction of quantum lattice vertex algebras based on the work presented in \cite{JKLT-Defom-va}, and presents the main result of this paper.

Throughout this paper, we denote by $\Z_+$ and $\N$ the set of positive and nonnegative integers, respectively.
For a vector space $W$ and $g(z)\in W[[z,z\inv]]$, we denote by $g(z)^+$ (resp. $g(z)^-$) the regular (singular) part of $g(z)$.

\section{Lattice vertex algebras}\label{sec:VAs}

A \emph{vertex algebra} (VA) is a vector space $V$ together with a \emph{vacuum vector} $\vac\in V$ and a vertex operator map
\begin{align}
    Y(\cdot,z):&V\longrightarrow \E(V):=\Hom(V,V((z)));\quad
    v\mapsto Y(v,z)=\sum_{n\in\Z}v_nz^{-n-1},
\end{align}
such that
\begin{align}\label{eq:vacuum-property}
    Y(\vac,z)v=v,\quad Y(v,z)\vac\in V[[z]],\quad \lim_{z\to 0}Y(v,z)\vac=v,
    \quad\te{for }v\in V,
\end{align}
and that
\begin{align}\label{eq:Jacobi}
    &z_0\inv\delta\(\frac{z_1-z_2}{z_0}\)Y(u,z_1)Y(v,z_2)
    -z_0\inv\delta\(\frac{z_2-z_1}{-z_0}\)Y(v,z_2)Y(u,z_1)\\
    &\quad=z_1\inv\delta\(\frac{z_2+z_0}{z_1}\)Y(Y(u,z_0)v,z_2)
    \quad\quad\te{for }u,v\in V.\nonumber
\end{align}

A \emph{module} of a VA $V$ is a vector space $W$ together with a vertex operator map
\begin{align}
    Y_W(\cdot,z):&V\longrightarrow \E(W);\quad
    v\mapsto Y_W(v,z)=\sum_{n\in\Z}v_nz^{-n-1},
\end{align}
such that $Y_W(\vac,z)=1_W$ and
\begin{align*}
    &z_0\inv\delta\(\frac{z_1-z_2}{z_0}\)Y_W(u,z_1)Y_W(v,z_2)
    -z_0\inv\delta\(\frac{z_2-z_1}{-z_0}\)Y_W(v,z_2)Y_W(u,z_1)\\
    &\quad=z_1\inv\delta\(\frac{z_2+z_0}{z_1}\)Y_W(Y(u,z_0)v,z_2)
    \quad\quad\te{for }u,v\in V.\nonumber
\end{align*}

In the rest of this section, we recall the construction and some results of lattice VAs (see \cite[Section 6.4-6.5]{LL}).
Let $Q$ be a non-degenerate even lattice of finite rank, i.e., a free abelian group of finite rank equipped with a non-degenerate symmetric $\Z$-valued bilinear form $\<\cdot,\cdot\>$ such that $\<\al,\al\>\in 2\Z$ for any $\al\in Q$.
Set
\begin{align*}
    \h=\C\ot_\Z Q
\end{align*}
and extend $\<\cdot,\cdot\>$ to a symmetric $\C$-valued bilinear form on $\h$.
We form an affine Lie algebra $\hat\h$, where
\begin{align*}
  \hat\h=\h\ot \C[t,t\inv]\oplus\C c
\end{align*}
as a vector space, and where $c$ is central and
\begin{align*}
  [\al(m),\beta(n)]=m\delta_{m+n,0}\<\al,\beta\>c
\end{align*}
for $\al,\beta\in\h,\,m,n\in\Z$ with $\al(m)$ denoting $\al\ot t^m$.
Define the following two abelian Lie subalgebras
\begin{align*}
  \hat\h^\pm=\h\ot t^{\pm 1}\C[t^{\pm 1}],
\end{align*}
and identify $\h$ with $\h\ot t^0$. Set
\begin{align*}
  \hat\h'=\hat\h^+\oplus\hat\h^-\oplus\C c
\end{align*}
which is a Heisenberg algebra. Then $\hat\h=\hat\h'\oplus \h$, a direct sum of Lie algebras.

Let
\begin{align*}
    Q^0=\set{\beta\in\h}{\<\beta,Q\>\in\Z}
\end{align*}
be the dual lattice of $Q$.
Choose a minimal $d\in\Z_+$, such that
\begin{align*}
    Q^0\subset (1/d)Q.
\end{align*}
Assume $Q=\oplus_{i\in I}\Z\al_i$.
We fix a total order $<$ on the index set $I$. Let $\epsilon:(1/d)Q\times (1/d)Q\to \C^\times$
be the bimultiplicative map defined by
\begin{align}
    \epsilon(1/d\al_i,1/d\al_j)=\begin{cases}
        1,&\mbox{if }i\le j,\\
        \exp(\pi\<\al_i,\al_j\>\sqrt{-1}/d^2),&\mbox{if }i>j.
    \end{cases}
\end{align}
Then $\epsilon$ is a $2$-cocycle of $(1/d)Q$ satisfying the following relations:
\begin{align*}
    \epsilon(\al,\beta)\epsilon(\beta,\al)\inv=(-1)^{\<\al,\beta\>},\quad \epsilon(\beta,0)=1=\epsilon(0,\beta)\quad \te{for }\al,\beta\in Q.
\end{align*}
Denote by $\C_\epsilon[(1/d)Q]$ the $\epsilon$-twisted group algebra of $(1/d)Q$, which by definition has a designated basis $\set{e_\al}{\al\in (1/d)Q}$ with relations
\begin{align}
    e_\al\cdot e_\beta=\epsilon(\al,\beta)e_{\al+\beta}\quad\te{for }\al,\beta\in (1/d)Q.
\end{align}
For any subset $S\subset (1/d)Q$, we denote
\begin{align*}
    \C_\epsilon[S]=\Span_\C\set{e_\al}{\al\in S}.
\end{align*}
Then $\C_\epsilon[Q]$ is a subalgebra of $\C_\epsilon[(1/d)Q]$, and $\C_\epsilon[Q^0]$ becomes a $\C_\epsilon[Q]$-module under the left multiplication action.
Moreover, for each coset $Q+\gamma\in Q^0/Q$, $\C_\epsilon[Q+\gamma]$ is a simple $\C_\epsilon[Q]$-submodule of $\C_\epsilon[Q^0]$,
and
\begin{align}
    \C_\epsilon[Q^0]=\bigoplus_{Q+\gamma\in Q^0/Q}\C_\epsilon[Q+\gamma]
\end{align}
as $\C_\epsilon[Q]$-modules.

Make $\C_{\epsilon}[Q^0]$ an $\hat \h$-module by letting $\hat\h'$
act trivially and letting $\h$ act by
\begin{eqnarray}
  h e_\be=\<h,\be\>e_\be \   \  \mbox{ for }h\in \h,\  \be\in Q^0.
\end{eqnarray}
Note that $S(\hat\h^-)$ is naturally an $\hat\h$-module of level $1$.
For each coset $Q+\gamma\in Q^0/Q$, we set
\begin{eqnarray}
V_{Q+\gamma}=S(\hat\h^-)\otimes \C_{\epsilon}[Q+\gamma],
\end{eqnarray}
the tensor product of $\hat \h$-modules, which is an $\hat \h$-module of level $1$.
Set
$$\vac=1\ot e_0\in  V_Q.$$
Identify $\h$ and $\C_{\epsilon}[Q+\gamma]$ as subspaces of $V_{Q+\gamma}$ via the correspondence
$$h\mapsto h(-1)\otimes 1\   (h\in\h) \quad\te{and}\quad e_\al\mapsto 1\otimes e_\al\   (\al\in Q+\gamma).$$
For $h\in \h$, set
\begin{align}
h(z)=\sum_{n\in\Z}h(n)z^{-n-1}.
\end{align}
On the other hand, for $\al\in Q$ set
\begin{align}\label{eq:def-E}
  E^\pm(\al,z)=\exp\(\sum_{n\in\Z_+} \frac{\al(\pm n)}{\pm n}z^{\mp n} \)
\end{align}
on $V_{Q+\gamma}$.
For $\al\in Q$,  define
$z^{\al}:\   \C_{\epsilon}[Q+\gamma]\rightarrow \C_{\epsilon}[Q+\gamma][z,z^{-1}]$ by
\begin{eqnarray}
z^\al\cdot e_\be=z^{\<\al,\be\>}e_\be\   \   \   \mbox{ for }\be\in Q+\gamma.
\end{eqnarray}

\begin{thm}\label{thm:lattice-VA}
There exists a VA structure on $V_Q$ (see \cite{Bor,FLM2}), which is uniquely determined by
the conditions that $\vac$ is the vacuum vector and that
%\begin{eqnarray}
%Y(\cdot,z):\  V_L\rightarrow (\te{End} V_{L^o})[[z,z^{-1}]]
%\end{eqnarray}
\begin{align*}
&Y_Q(h,z)=h(z),\quad
Y_Q(e_\al,z)=E^-(-\al,z)E^+(-\al,z)e_\al z^\al\quad\te{for } h\in\h,\,\,\al\in Q.
\end{align*}
Moreover, every $V_Q$-module is completely reducible, and for each simple $V_Q$-module $W$,
there exists a coset $Q+\gamma\in Q^0/Q$, such that $W\cong V_{Q+\gamma}$ (see \cite{D1}, \cite[Theorem 3.16]{DLM}).
\end{thm}

\section{Quantum vertex algebras}\label{sec:qvas}

In this paper, we let $\hbar$ be a formal variable, and let $\C[[\hbar]]$ be the ring of formal power series in $\hbar$.
A $\C[[\hbar]]$-module $V$ is \emph{topologically free} if $V=V_0[[\hbar]]$
for some vector space $V_0$ over $\C$.
It is known that a $\C[[\hbar]]$-module $W$ is topologically free if and only if it is Hausdorff complete under the $\hbar$-adic topology and
\begin{align}\label{eq:torsion-free}
  \hbar w=0\quad\te{implies}\quad w=0\quad\te{for any }w\in W
\end{align}
(see \cite{Kassel-topologically-free}).
For another topologically free $\C[[\hbar]]$-module $U=U_0[[\hbar]]$, we recall the complete tensor product
\begin{align*}
    U\wh\ot V=(U_0\ot V_0)[[\hbar]].
\end{align*}

We view a vector space as a $\C[[\hbar]]$-module by letting $\hbar=0$.
Fix a $\C[[\hbar]]$-module $W$. For $k\in\Z_+$, and some formal variables $z_1,\dots,z_k$, we define
\begin{align}
    \E^{(k)}(W;z_1,\dots,z_k)=\Hom_{\C[[\hbar]]}\(W,W((z_1,\dots,z_k))\)
\end{align}
We will denote $\E^{(k)}(W;z_1,\dots,z_k)$ by $\E^{(k)}(W)$ if there is no ambiguity,
and will denote $\E^{(1)}(W)$ by $\E(W)$.
An ordered sequence $(a_1(z),\dots,a_k(z))$ in $\E(W)$ is said to be \emph{compatible} (\cite{Li-nonlocal})
if there exists an $m\in\Z_+$, such that
\begin{align*}
    \(\prod_{1\le i<j\le k}(z_i-z_j)^m\)a_1(z_1)\cdots a_k(z_k)\in\E^{(k)}(W).
\end{align*}

Now, we assume that $W=W_0[[\hbar]]$ is topologically free.
Define
\begin{align}
    \E_\hbar^{(k)}(W;z_1,\dots,z_k)=\Hom_{\C[[\hbar]]}\(W,W_0((z_1,\dots,z_k))[[\hbar]]\).
\end{align}
Similarly, we denote $\E_\hbar^{(k)}(W;z_1,\dots,z_k)$ by $\E_\hbar^{(k)}(W)$ if there is no ambiguity,
and denote $\E_\hbar^{(1)}(W)$ by $\E_\hbar(W)$ for short.
We note that $\E_\hbar^{(k)}(W)=\E^{(k)}(W_0)[[\hbar]]$ is topologically free.
For $n,k\in\Z_+$, the quotient map from $W$ to $W/\hbar^nW$ induces the following $\C[[\hbar]]$-module map
\begin{align*}
    \wt\pi_n^{(k)}:\End_{\C[[\hbar]]}(W)[[z_1^{\pm 1},\dots,z_k^{\pm 1}]]
    \to \End_{\C[[\hbar]]}(W/\hbar^nW)[[z_1^{\pm 1},\dots,z_k^{\pm 1}]].
\end{align*}
For $A(z_1,z_2),B(z_1,z_2)\in\Hom_{\C[[\hbar]]}(W,W_0((z_1))((z_2))[[\hbar]])$, we write $A(z_1,z_2)\sim B(z_2,z_1)$
if for each $n\in\Z_+$ there exists $k\in\N$, such that
\begin{align*}
    (z_1-z_2)^k\wt\pi_n^{(2)}(A(z_1,z_2))=(z_1-z_2)^k\wt\pi_n^{(2)}(B(z_2,z_1)).
\end{align*}
Let $Z(z_1,z_2):\E_\hbar(W)\wh\ot \E_\hbar(W)\wh\ot\C((z))[[\hbar]]\to \End_{\C[[\hbar]]}(W)[[z_1^{\pm 1},z_2^{\pm 1}]]$
be defined by
\begin{align*}
    Z(z_1,z_2)(a(z)\ot b(z)\ot f(z))=\iota_{z_1,z_2}f(z_1-z_2)a(z_1)b(z_2).
\end{align*}

For each $k\in\Z_+$, the inverse system
\begin{align*}
    \xymatrix{
    0&W/\hbar W\ar[l]&W/\hbar^2W\ar[l]&W/\hbar^3W\ar[l]&\cdots\ar[l]
    }
\end{align*}
induces the following inverse system
\begin{align}\label{eq:E-h-inv-sys}
    \xymatrix{
    0&\E^{(k)}(W/\hbar W)\ar[l]&\E^{(k)}(W/\hbar^2W)\ar[l]&\cdots\ar[l]
    }
\end{align}
Then $\E_\hbar^{(k)}(W)$ is isomorphic to the inverse limit of \eqref{eq:E-h-inv-sys}.
The map $\wt \pi_n^{(k)}$ induces a $\C[[\hbar]]$-module $\pi_n^{(k)}:\E_\hbar^{(k)}(W)\to \E^{(k)}(W/\hbar^nW)$.
It is easy to verify that $\ker \pi_n^{(k)}=\hbar^n\E_\hbar^{(k)}(W)$.
We will denote $\pi_n^{(1)}$ by $\pi_n$ for short.
An ordered sequence $(a_1(z),\dots,a_r(z))$ in $\E_\hbar(W)$ is called \emph{$\hbar$-adically compatible} if for every $n\in\Z_+$,
the sequence $$(\pi_n(a_1(z)),\dots,\pi_n(a_r(z)))$$ in $\E(W/\hbar^nW)$ is compatible.
A subset $U$ of $\E_\hbar(W)$ is called \emph{$\hbar$-adically compatible} if every finite sequence in $U$ is $\hbar$-adically compatible.
Let $(a(z),b(z))$ in $\E_\hbar(W)$ be $\hbar$-adically compatible. That is, for any $n\in\Z_+$, we have
\begin{align*}
    (z_1-z_2)^{k_n}\pi_n(a(z_1))\pi_n(b(z_2))\in\E^{(2)}(W/\hbar^nW)\quad\te{for some }k_n\in\N.
\end{align*}
We recall the following vertex operator map (\cite{Li-h-adic}):
\begin{align*}
    &Y_\E(a(z),z_0)b(z)=\sum_{n\in\Z}a(z)_nb(z)z_0^{-n-1}\\
    =&\varinjlim_{n>0}z_0^{-k_n}\left.\((z_1-z)^{k_n}\pi_n(a(z_1))\pi_n(b(z))\)\right|_{z_1=z+z_0}.
\end{align*}

An \emph{$\hbar$-adic nonlocal vertex algebra} (\cite{Li-h-adic}) is a topologically free $\C[[\hbar]]$-module $V$ equipped with a vacuum vector $\vac$ such that the vacuum property \eqref{eq:vacuum-property} holds, and a vertex operator map $Y(\cdot,z):V\to \E_\hbar(V)$
such that
\begin{align*}
    \set{Y(u,z)}{u\in V}\subset\E_\hbar(V)\quad\te{is $\hbar$-adically compatible},
\end{align*}
and that
\begin{align*}
    Y_\E(Y(u,z),z_0)Y(v,z)=Y(Y(u,z_0)v,z)\quad\te{for }u,v\in V.
\end{align*}
We denote by $\partial$ the canonical derivation of $V$:
\begin{align}
    u\to\partial u=\lim_{z\to 0}\frac{d}{dz}Y(u,z)\vac.
\end{align}
Moreover, a \emph{$V$-module} is a topologically free $\C[[\hbar]]$-module $W$ equipped with a vertex operator map $Y_W(\cdot,z):V\to \E_\hbar(W)$, such that $Y_W(\vac,z)=1_W$,
\begin{align*}
    \set{Y_W(u,z)}{u\in V}\subset\E_\hbar(W)\quad\te{is $\hbar$-adically compatible},
\end{align*}
and that
\begin{align*}
    Y_\E(Y_W(u,z),z_0)Y_W(v,z)=Y_W(Y(u,z_0)v,z)\quad\te{for }u,v\in V.
\end{align*}

An \emph{$\hbar$-adic quantum VA} is an $\hbar$-adic nonlocal VA $V$ equipped with a $\C[[\hbar]]$-module map (called a \emph{quantum Yang-Baxter operator})
\begin{align}
  S(z):V\wh\ot V\to V\wh\ot V\wh\ot \C((z))[[\hbar]],
\end{align}
which satisfies the \emph{shift condition}:
\begin{align}\label{eq:qyb-shift}
  [\partial\ot 1,S(z)]=-\frac{d}{dz}S(z),
\end{align}
the \emph{quantum Yang-Baxter equation}:
\begin{align}\label{eq:qyb}
  S^{12}(z_1)S^{13}(z_1+z_2)S^{23}(z_2)=S^{23}(z_2)S^{13}(z_1+z_2)S^{12}(z_1),
\end{align}
and the \emph{unitarity condition}:
\begin{align}\label{eq:qyb-unitary}
  S^{21}(z)S(-z)=1,
\end{align}
satisfying the following conditions:

  (1) The \emph{vacuum property}:
  \begin{align}\label{eq:qyb-hex1}
    S(z)(\vac\ot v)=\vac\ot v,\quad \te{for }v\in V.
  \end{align}

 (2) The \emph{$S$-locality}:
  For any $u,v\in V$, one has
  \begin{align}\label{eq:qyb-locality}
  Y(u,z_1)Y(v,z_2)\sim Y(z_2)(1\ot Y(z_1))S(z_2-z_1)(v\ot u).
  \end{align}

  (3) The \emph{hexagon identity}:
  \begin{align}\label{eq:qyb-hex-id}
    S(z_1)Y^{12}(z_2)=Y^{12}(z_2)S^{23}(z_1)S^{13}(z_1+z_2).
  \end{align}

\begin{rem}{\em
Let $(V,S(z))$ be a quantum VA.
Then
\begin{align}
  &S(z)(v\ot\vac)=v\ot\vac\quad \te{for }v\in V,\label{eq:qyb-vac-id-alt}\\
  &[1\ot\partial,S(z)]=\frac{d}{dz}S(z),\label{eq:qyb-shift-alt}\\
  &S(z_1)(1\ot Y(z_2))=(1\ot Y(z_2))S^{12}(z_1-z_2)S^{13}(z_1)\label{eq:qyb-hex2},\\
  &S(z)f(\partial\ot 1)=f\(\partial\ot 1+\pd{z}\)S(z)\quad \te{for }f(z)\in\C[z][[\hbar]],\label{eq:qyb-shift-total1}\\
  &S(z)f(1\ot \partial)=f\(1\ot \partial-\pd{z}\)S(z)\quad \te{for }f(z)\in\C[z][[\hbar]].\label{eq:qyb-shift-total2}
\end{align}
}
\end{rem}

\begin{rem}\label{rem:Jacobi-S}
\emph{
Let $(W,Y_W)$ be a $V$-module.
Then for $u,v\in V$, we have that
\begin{align*}
    &Y_W(u,z_1)Y_W(v,z_2)-Y_W(z_2)(1\ot Y_W(z_1))S(z_2-z_1)(v\ot u)\\
    &\quad=Y_W(Y(u,z_1-z_2)^-v-Y(u,-z_2+z_1)^-v,z_2).
\end{align*}
}
\end{rem}

For a topologically free $\C[[\hbar]]$-module $V$ and a submodule $U\subset V$, we denote by $\bar U$ the closure of $U$ and set
\begin{align}
  [U]=\set{u\in V}{\hbar^n u\in U\,\,\te{for some }n\in\Z_+}.
\end{align}
Then $\overline{[U]}$ is the minimal closed submodule that is invariant under the operation $[\cdot]$.
Suppose further that $V$ is an $\hbar$-adic nonlocal VA, and $\vac\in U$.
We set
\begin{align}
  U'=\Span_{\C[[\hbar]]}\set{u_mv}{u,v\in U,\,m\in\Z},\quad\te{and}\quad \wh U=\overline{[U']}.
\end{align}
For a subset $S\subset V$, we define
\begin{align*}
  S^{(1)}=\overline{\left[\Span_{\C[[\hbar]]}S\cup\{\vac\}\right]},\quad\te{and}\quad
  S^{(n+1)}=\wh{S^{(n)}}\quad\te{for }n\ge 1.
\end{align*}
Then we have
\begin{align*}
  \vac\in S^{(1)}\subset S^{(2)}\subset\cdots.
\end{align*}
Set
\begin{align*}
  \<S\>=\overline{\left[\cup_{n\ge 1}S^{(n)}\right]}.
\end{align*}
Then $\<S\>$ is the unique minimal closed $\hbar$-adic nonlocal subVA of $V$, such that $[\<S\>]=\<S\>$.
In addition, we say that $V$ is generated by $S$ if $\<S\>=V$.

\begin{rem}\label{rem:top-eq}\emph{
Let $V$ be a topologically free $\C[[\hbar]]$-module, and let $U$ be a submodule of $V$.
There are two topologies on $U$: the $\hbar$-adic topology of $U$ itself and the induced topology.
If $[U]=U$, then such two topologies are equivalent (\cite[Lemma 3.5]{Li-h-adic}).
More precisely, for each $n\in\N$, one has that $U\cap \hbar^n V=\hbar^n U$.
}
\end{rem}

\begin{convention}
Let $V$ be an $\hbar$-adic nonlocal VA, and let $W$ be a $V$-module.
In this paper, the submodule $W_1$ of $W$ that we consider is required to meet the conditions that $[W_1]=W_1$ and $\overline{W_1}=W_1$.
\end{convention}

\begin{lem}\label{lem:simple}
Let $V_0$ be a nonlocal VA, and let $W_0$ be a simple $V_0$-module.
View $V=V_0[[\hbar]]$ as an $\hbar$-adic nonlocal VA, and $W=W_0[[\hbar]]$ as a $V$-module.
Then $W$ is simple.
\end{lem}

\begin{proof}
Let $W'$ be a nonzero submodule of $W$.
Since $W'$ is closed and $W$ is topologically free, we have that $W'$ is also topologically free.
As $[W']=W'$, we get that
\begin{align*}
  &(W'+\hbar W)/\hbar W\cong W'/(W'\cap \hbar W)=W'/\hbar W'\ne 0.
\end{align*}
So $(W'+\hbar W)/\hbar W\subset W/\hbar W$ is a nonzero $V_0$-submodule.
Then $(W'+\hbar W)/\hbar W=W/\hbar W$, since $W/\hbar W\cong W_0$ is simple.
It implies that $W=W'+\hbar W$.
For each $w_0\in W$, we have that
\begin{align*}
  w_0=& w_0'+\hbar w_1\quad\te{for some }w_0'\in W',\,\,w_1\in W\\
  =&w_0'+\hbar (w_1'+\hbar w_2) \quad\te{for some }w_1'\in W',\,\,w_2\in W\\
  =&\cdots=\sum_{n=0}^\infty w_n'\hbar^n\in \overline{W'}=W'.
\end{align*}
It shows that $W=W'$. Therefore, $W$ is simple.
\end{proof}

\begin{lem}\label{lem:submodule}
Let $V_0$ be a nonlocal VA, let $W_0$ be a $V_0$-module and let $W_0'$ be a submodule of $W_0$.
View $V=V_0[[\hbar]]$ as an $\hbar$-adic nonlocal VA.
Then $W_0'[[\hbar]]$ is a $V$-submodule of $W_0[[\hbar]]$.
\end{lem}

\begin{proof}
We denote $W_0'[[\hbar]]$ by $W'$.
It suffices to demonstrate that $\big[W_0'[[\hbar]]\big]=W_0'[[\hbar]]$.
Consider $w\in W_0[[\hbar]]$ with $\hbar w\in W_0'[[\hbar]]$.
Thus, $\hbar w$ can be expressed as:
\begin{align*}
  \hbar w=\sum_{n=0}^\infty w_n\hbar^n
\end{align*}
for some $w_n\in W_0'$ ($n\in \N$).
This implies that
\begin{align*}
  &w_0=\hbar w-\sum_{n=1}^\infty w_n\hbar^n\in W_0'\cap \hbar W_0[[\hbar]]=0.
\end{align*}
Consequently, $w=\sum_{n=1}^\infty w_n\hbar^{n-1}\in W_0'[[\hbar]]$, as required.
\end{proof}

\begin{lem}\label{lem:simple2}
Let $V_0$ be a nonlocal VA, such that every $V_0$-module is completely reducible.
View $V=V_0[[\hbar]]$ as an $\hbar$-adic nonlocal VA.
Then for each simple $V$-module $W$,
there is a simple $V_0$-module $W_0$ such that $W\cong W_0[[\hbar]]$.
\end{lem}

\begin{proof}
Note that $W$ is completely reducible as a $V_0$-module and $\hbar W$ is a $V_0$-submodule of $W$.
There is a $V_0$-submodule $W_0$ of $W$ such that $W=W_0\oplus \hbar W$.
Since $W$ is topologically free and $W/\hbar W\cong W_0$, we get that $W=W_0[[\hbar]]$.
To complete the proof, we only need to show that $W_0$ is a simple $V_0$-module.
Let $W_0'$ be a nonzero $V_0$-submodule of $W_0$.
By Lemma \ref{lem:submodule}, $W_0'[[\hbar]]$ is a nonzero $V$-submodule of $W$.
Since $W$ is simple, it follows that $W=W_0'[[\hbar]]$.
This implies that $W_0=W_0'$, thereby demonstrating that $W_0$ is simple.
\end{proof}

\begin{lem}\label{lem:completely-reducible}
Let $V_0$ be a nonlocal VA, such that every $V_0$-module is completely reducible.
View $V=V_0[[\hbar]]$ as an $\hbar$-adic nonlocal VA.
Then every $V$-module is completely reducible.
\end{lem}

\begin{proof}
Let $W$ be a $V$-module and let $W'$ be a submodule of $W$.
Note that both $W'$ and $W$ are completely reducible as $V_0$-modules.
We choose a $V_0$-submodule $W_0'$ of $W'$ such that $W'=W_0'\oplus \hbar W'$,
and another $V_0$-submodule $W_0''$ of $W$ such that $W=W_0''\oplus (W_0'+\hbar W)$.
As a $\C[[\hbar]]$-submodule of $W$, $W'$ is Hausdorff under the $\hbar$-adic topology and satisfies condition \eqref{eq:torsion-free}. In addition, $W'$ is
complete by assumption.
Then $W'$ is topologically free and $W'=W_0'[[\hbar]]$.
Noticing that $W'\cap \hbar W=\hbar W'$ by Remark \ref{rem:top-eq}, we have that $W_0'\cap \hbar W=W_0'\cap W'\cap \hbar W=W_0'\cap \hbar W'=0$.
Then $W=W_0'\oplus W_0''\oplus \hbar W$.
Since $W$ is topologically free, we have that $W=(W_0'\oplus W_0'')[[\hbar]]$
and $W=W'\oplus W_0''[[\hbar]]$.
By utilizing Lemma \ref{lem:submodule}, we have that $W_0''[[\hbar]]$ is the required $V$-submodule of $W$.
\end{proof}

We list some technical results that will be used later on.
\begin{prop}\label{prop:vacuum-like}
Let $(V,Y,\vac)$ be an $\hbar$-adic quantum VA with quantum Yang-Baxter operator $S(z)$,
and let $(W,Y_W)$ be a $V$-module.
Assume that $U$ is a subset of $V$, such that
\begin{align}\label{eq:s-closed}
  S(z)(U \wh\ot U)\subset U\wh\ot U\wh\ot \C((z))[[\hbar]]
\end{align}
and $V$ is generated by $U$.
Suppose further that there exists $w^+\in W$, such that
\begin{align}\label{eq:vacuum-like}
  Y_W(u,z)w^+\in W[[z]]\quad\te{for }u\in U.
\end{align}
Then the relation \eqref{eq:vacuum-like} holds for all $u\in V$.
\end{prop}

\begin{proof}
We note that for any subset $S$ of $V$ such that the two relations \eqref{eq:s-closed} and \eqref{eq:vacuum-like} hold for $S$, these two relations hold for
$S\cup \{\vac\}$, $\Span_{\C[[\hbar]]}S$, $[S]$, and $\bar S$.
It is easy to see that \eqref{eq:s-closed} and \eqref{eq:vacuum-like} hold for $U^{(1)}$.
Suppose these two relations hold for $U^{(n)}$.
It is immediate from \eqref{eq:qyb-hex1} and \eqref{eq:qyb-hex2} that the relation \eqref{eq:s-closed} holds for $U^{(n+1)}$.
Next we show that \eqref{eq:vacuum-like} also holds for all $u\in U^{(n+1)}$.
Let $u,v\in U^{(n)}$. From the $S$-locality, we get that for each $n\ge 0$, there exists $k\ge 0$, such that
\begin{align*}
  &(z_1-z_2)^kY_W(u,z_1)Y_W(v,z_2)w^+
  \equiv (z_1-z_2)^kY_W^{12}(z_2)Y_W^{23}(z_1)S^{12}(z_2-z_1)(v\ot u\ot w^+)
\end{align*}
modulo $\hbar^n W$.
From the relation \eqref{eq:vacuum-like}, we have that
\begin{align*}
  (z_1-z_2)^kY_W(u,z_1)Y_W(v,z_2)w^+\in W[[z]]+\hbar^n W[[z,z\inv]].
\end{align*}
Then from the weak associativity, we get that
\begin{align*}
  Y_W(Y(u,z_0)v,z_2)w^+\equiv z_0^{-k}\((z_1-z_2)^kY_W(u,z_1)Y_W(v,z_2)w^+\)|_{z_1=z_2+z_0}
\end{align*}
lies in $W[[z_2]]((z_0))+\hbar^n W[[z_0^{\pm 1},z_2^{\pm 1}]]$.
Since $n$ is arbitrary, we have that
\begin{align*}
  Y_W(Y(u,z_0)v,z_2)w^+\in W[[z_2,z_0,z_0\inv]].
\end{align*}
Therefore, the relation \eqref{eq:vacuum-like} holds for $u\in \(U^{(n)}\)'$,
and hence holds for $u\in \wh{U^{(n)}}=U^{(n+1)}$.
As
\begin{align*}
  U^{(n)}\subset U^{(n+1)}\quad\te{for }n\ge 1,\quad\te{and}\quad
  V=\<U\>=\overline{\left[ \cup_{n\ge 1}U^{(n)} \right]},
\end{align*}
we get that the relation \eqref{eq:vacuum-like} holds for all $u\in V$.
\end{proof}

\begin{prop}\label{prop:vacuum-like-sp}
Let $(U,Y_U,\vac_U)$ and $(V,Y_V,\vac_V)$ be two $\hbar$-adic nonlocal VAs,
and let $S$ be a $\C[[\hbar]]$-submodule of $U$ such that $U$ is generated by $S$.
Assume that there exists a $U$-module structure $Y_V^U$ on $V$ such that
\begin{align}\label{eq:vacuum-like-total}
  Y_V^U(u,z)\vac_V\in V[[z]].
\end{align}
Suppose that $\psi^0:S\to V$ is a $\C[[\hbar]]$-module map such that
\begin{align}\label{eq:Y-V-U=Y-V-psi}
  Y_V^U(s,z)=Y_V\(\psi^0(s),z\)\quad\te{for }s\in S.
\end{align}
%Define $\psi:U\to V$ by
%\begin{align}
%  &u\mapsto \lim_{z\to 0}Y_V^U(u,z)\vac_V.
%\end{align}
%Then $\psi$ is an $\hbar$-adic nonlocal VA homomorphism
Then there exists an $\hbar$-adic nonlocal VA homomorphism $\psi:U\to V$, such that
\begin{align*}
  \psi(s)=\psi^0(s)\quad\te{for }s\in S.
\end{align*}
\end{prop}

\begin{proof}
From \eqref{eq:vacuum-like-total}, we define $\psi:U\to V$ by
\begin{align*}
  \psi(u)=\lim_{z\to 0} Y_V^U(u,z)\vac_V\quad\te{for }u\in U.
\end{align*}
We deduce from \eqref{eq:Y-V-U=Y-V-psi} that
\begin{align*}
  \psi(s)=\lim_{z\to 0}Y_V^U(s,z)\vac_V=\lim_{z\to 0}Y_V(\psi^0(s),z)\vac_V=\psi^0(s)\quad\te{for }s\in S.
\end{align*}
Let $u,v\in U$, and let $n$ be an arbitrary non-negative integer.
From the weak associativity, we get that there exists $k\ge 0$, such that
\begin{align}\label{eq:vacuum-like-sp-temp1}
  &(z_0+z_2)^kY_V^U(u,z_0+z_2)Y_V^U(v,z_2)\vac_V\nonumber\\
  \equiv &(z_0+z_2)^kY_V^U(Y_U(u,z_0)v,z_2)\vac_V \mod \hbar^n V[[z_0^{\pm 1},z_2^{\pm 1}]].
\end{align}
From \eqref{eq:vacuum-like-total}, we get that
\begin{align*}
  Y_V^U(Y_U(u,z_0)v,z_2)\vac_V\in V((z_0,z_2))+\hbar^n V[[z_0^{\pm 1},z_2^{\pm 1}]].
\end{align*}
Combining this with \eqref{eq:vacuum-like-sp-temp1}, we get that
\begin{align*}
  Y_V^U(u,z_0+z_2)Y_V^U(v,z_2)\vac_V\equiv Y_V^U(Y_U(u,z_0)v,z_2)\vac_V \mod \hbar^n V[[z_0^{\pm 1},z_2^{\pm 1}]].
\end{align*}
Since the choice of $n$ is arbitrary, we have that
\begin{align}
  Y_V^U(u,z_0+z_2)Y_V^U(v,z_2)\vac_V= Y_V^U(Y_U(u,z_0)v,z_2)\vac_V.
\end{align}
Taking $z_2\to 0$, we get that
\begin{align*}
  Y_V^U(u,z_0)\psi(v)=\psi\(Y_U(u,z_0)v\).
\end{align*}
For any $s\in S$ and $v\in U$, we have that
\begin{align*}
  Y_V(\psi(s),z)\psi(v)=Y_V(\psi^0(s),z)\psi(v)=Y_V^U(s,z)\psi(v)=\psi\(Y_U(u,z)v\).
\end{align*}
Since $U$ is generated by $S$, $\psi$ is an $\hbar$-adic nonlocal VA homomorphism.
\end{proof}

\section{Quantum lattice vertex algebras}\label{sec:qlattice}

We first recall from \cite{Li-smash} the notions and results on smash product $\hbar$-adic nonlocal VAs.
\begin{de}
An {\em $\hbar$-adic nonlocal vertex bialgebra} is an $\hbar$-adic nonlocal vertex algebra $H$ with an ($\hbar$-adic)
coalgebra structure $(\Delta,\epsilon)$ on $H$ such that both $\Delta: H\rightarrow H\wh\otimes H$ and
$\epsilon: H\rightarrow \C[[\hbar]]$ are homomorphisms of $\hbar$-adic nonlocal vertex algebras.
\end{de}

\begin{de}
Let $H$ be an $\hbar$-adic nonlocal vertex bialgebra.
An {\em $H$-module $\hbar$-adic nonlocal VA} is an $\hbar$-adic nonlocal VA $V$
with a module structure $\tau$ on $V$ for $H$ viewed as an $\hbar$-adic nonlocal vertex algebra such that
\begin{align}
  &\tau(a,x)v\in V\wh\otimes \C((z))[[\hbar]],\quad
  %&Y_{V}^H(h,x)v\in V\ot \C((x)),\label{eq:mod-va-for-vertex-bialg1}\\
  \tau(a,z)\vac_V=\varepsilon(a)\vac_V,\label{eq:mod-va-for-vertex-bialg2}\\
  &\tau(a,z_1)Y(u,z_2)v=\sum Y(\tau(a_{(1)},z_1-z_2)u,z_2)\tau(a_{(2)},z_1)v
  \label{eq:mod-va-for-vertex-bialg3}
\end{align}
for $a\in H$, $u,v\in V$, where $\vac_V$ denotes the vacuum vector of $V$
and $\Delta(a)=\sum a_{(1)}\ot a_{(2)}$ is the coproduct in the Sweedler notation.
\end{de}

\begin{thm}\label{thm:smash}
Let $H$ be an $\hbar$-adic nonlocal vertex bialgebra, and let $(V,\tau)$ be an $H$-module $\hbar$-adic nonlocal VA.
Then $V\sharp H=V\wh\ot H$ carries an $\hbar$-adic nonlocal VA structure with vertex operator map defined by
\begin{align*}
  Y^\sharp(u\ot h,z)(v\ot k)=\sum Y(u,z)\tau(h_{(1)},z)v\ot Y(h_{(2)},z)k\quad\te{for }u,v\in V,\,h,k\in H.
\end{align*}
Moreover, let $W$ be a $V$-module and an $H$-module such that
\begin{align*}
    &Y_W(h,z)w\in W\wh\ot \C((z))[[\hbar]],\\
    &Y_W(h,z_1)Y_W(v,z_2)w=\sum Y_W(\tau(h_{(1)},z_1-z_2)v,z_2)Y_W(h_{(2)},z_1)w
\end{align*}
for $h\in S$, $v\in V$, $w\in W$, where $S$ is a generating subset of $H$ as an $\hbar$-adic nonlocal VA.
Then $W$ is a module for $V\sharp H$ with
\begin{align*}
    Y_W^\sharp(v\ot h,z)w=Y_W(v,z)Y_W(h,z)w\quad \te{for }h\in H,v\in V,w\in W.
\end{align*}
\end{thm}

The following notions and results were given in \cite{JKLT-Quantum-lattice-va}, which are straightforward $\hbar$-adic analogues of the ones given in \cite{JKLT-Defom-va}.

\begin{de}
Let $H$ be an $\hbar$-adic nonlocal vertex bialgebra.  A {\em right $H$-comodule $\hbar$-adic nonlocal vertex algebra} is
an $\hbar$-adic nonlocal vertex algebra $V$ equipped with an $\hbar$-adic nonlocal vertex algebra homomorphism
 $\rho: V\rightarrow V\wh\ot H$  such that
\begin{align}\label{eq:comod-cond}
(\rho\ot 1)\rho=(1\ot\Delta)\rho,\quad
  (1\ot \epsilon)\rho=\te{id}_V,
\end{align}
i.e., $\rho$ is also a right comodule structure on $V$ for $H$ viewed as a coalgebra.
\end{de}

\begin{de}\label{compatible-ma-cma}
Let $H$ be an $\hbar$-adic nonlocal vertex bialgebra and let $V$ be a right $H$-comodule $\hbar$-adic nonlocal vertex algebra
with comodule structure map $\rho:V\rightarrow V\wh\ot H$.
Denote by $\mathfrak L^\rho_H(V)$ the set consisting of each $\C[[\hbar]]$-linear map
\begin{align}
\tau(\cdot,z):\  H\rightarrow \Hom(V,V\wh\ot\C((z))[[\hbar]])
\end{align}
such that $V$ with $\tau(\cdot,z)$ is an $H$-module $\hbar$-adic nonlocal vertex algebra and
$\rho$ is an $H$-module homomorphism with $H$ acting on the first factor of $V\wh\otimes H$ only, i.e.,
\begin{align}\label{compatible-relation}
  \rho(\tau(a,z)v)=(\tau(a,z)\ot 1)\rho(v)\quad\te{for }a\in H,\  v\in V.
\end{align}
\end{de}

\begin{thm}\label{thm:deform-va-h}
Let $H$ be a cocommutative $\hbar$-adic nonlocal vertex bialgebra,
let $(V,\rho)$ be a right $H$-comodule $\hbar$-adic nonlocal vertex algebra,
and let $\tau\in \mathfrak L^\rho_H(V)$.
Define
\begin{align}
\mathfrak D_{\tau}^\rho (Y)(\cdot,z):V\to (\te{End }V)[[z,z^{-1}]];\quad v\mapsto\sum Y(v_{(1)},z)\tau(v_{(2)},z),
\end{align}
where $\rho(v)=\sum v_{(1)}\ot v_{(2)}\in V\wh\ot H$.
Then $(V,\mathfrak D_{\tau}^\rho (Y),\vac)$ carries the structure of an $\hbar$-adic nonlocal vertex algebra,
which is denoted by $\mathfrak D_{\tau}^\rho (V)$.
Furthermore,  $\rho$ is an $\hbar$-adic nonlocal VA homomorphism from $\mathfrak D_\tau^\rho(V)$ to $V\sharp H$.
\end{thm}

\begin{rem}
{\em
Let $H$ be a cocommutative $\hbar$-adic nonlocal vertex bialgebra,
and let $(V,\rho)$ be a right $H$-comodule $\hbar$-adic nonlocal vertex algebra.
The counit $\varepsilon$ of $H$ defines an element in $\mathfrak L_H^\rho(V)$, which is still denoted by $\varepsilon$.
This element $\varepsilon$ is defined by
\begin{align*}
\varepsilon(a,z)v=\varepsilon(a)v\    \   \  \mbox{ for }a\in H,\  v\in V.
\end{align*}
Moreover
$\mathfrak D_{\varepsilon}^\rho (V)=V$.}
\end{rem}

\begin{de}
Let $H$ be a cocommutative $\hbar$-adic nonlocal vertex bialgebra,
and let $(V,\rho)$ be a right $H$-comodule $\hbar$-adic nonlocal vertex algebra.
An element $\tau\in\mathfrak L_H^\rho(V)$ is called \emph{invertible}, if there exists another $\tau'\in\mathfrak L_H^\rho(V)$, such that
\begin{align*}
  [\tau(h,z_1),\tau'(k,z_2)]=0\quad\te{for }h,k\in H,
\end{align*}
and $\tau\ast \tau'=\varepsilon$, where
\begin{align*}
  (\tau\ast\tau')(h,z)=\sum \tau(h_{(1)},z)\tau'(h_{(2)},z)\quad\te{for }h\in H.
\end{align*}
We denote $\tau'$ by $\tau\inv$.
\end{de}

\begin{thm}\label{thm:S-op}
Let $H$ be a cocommutative $\hbar$-adic vertex bialgebra and let $(V,\rho)$ be a right $H$-comodule $\hbar$-adic vertex algebra.
Assume $\tau$ is an invertible element of $\mathfrak L_H^\rho(V)$  with inverse $\tau\inv$.
Then $\mathfrak D_\tau^\rho(V)$ is an $\hbar$-adic quantum VA with quantum Yang-Baxter operator $S(z)$ defined by
\begin{align*}
  S(z)(v\ot u)=\sum \tau(u_{(2)},-z)v_{(1)}\ot \tau\inv (v_{(2)},z)u_{(1)}\quad\te{for }u,v\in V.
\end{align*}
\end{thm}

Now, we recall the construction of quantum lattice VAs introduced in \cite{JKLT-Quantum-lattice-va}, which is also a straightforward $\hbar$-adic analogue of the one given in \cite{JKLT-Defom-va}. Set
\begin{align}
    B_Q=\(S(\hat\h^-)\ot \C[Q]\)[[\hbar]].
\end{align}
It is a bialgebra over $\C[[\hbar]]$, where the coproduct $\Delta$ and the counit $\varepsilon$ are defined by
\begin{align*}
  &\Delta(h(-n))=h(-n)\ot 1+1\ot h(-n),\quad \Delta(e^\al)=e^\al\ot e^\al,\quad
  \varepsilon(h(-n))=0,\quad \varepsilon(e^\al)=1
\end{align*}
for $h\in\h$, $n\in\Z_+$, $\al\in Q$.
Moreover, $B_Q$ admits a derivation $\partial$ which is defined by
\begin{align*}
  \partial(h(-n))=nh(-n-1),\quad \partial(u\ot e^\al)=\al(-1)u\ot e^\al+\partial u\ot e^\al
\end{align*}
for $h\in\h$, $n\in\Z_+$, $u\in S(\hat\h^-)$ and $\al\in Q$.
We identify $\h$ (resp. $\C[Q]$) with $\h(-1)\ot 1\subset B_Q$ (resp. $1\ot \C[Q]\subset B_Q$).
Then $B_Q$ becomes a cocommutative and commutative $\hbar$-adic vertex bialgebra with the vertex operator map $Y_{B_Q}(\cdot,z)$ determined by
\begin{align*}
  &Y_{B_Q}(h,z)=\sum_{n>0}h(-n)z^{n-1},\quad Y_{B_Q}(e^\al,z)=e^\al E^+(\al,z)
\end{align*}
for $h\in\h$ and $\al\in Q$.
We view $V_Q[[\hbar]]$ as an $\hbar$-adic VA in the obvious way.

For a $\C[[\hbar]]$-module $W$, and a linear map $\varphi(\cdot,z):\h\to W[[z,z\inv]]$,
%Let $W$ be a topologically free $\C[[\hbar]]$-module.
%For $h\in\h$, $f(z)\in\C((z))[[\hbar]]$ and a linear map $\varphi(\cdot,z):\h\to \E_\hbar(W)$,
we define (see \cite{EK-qva})
\begin{align}\label{eq:def-Phi}
  \Phi(h\ot f(z),\varphi)=\Res_{z_1}\varphi(h,z_1)f(z-z_1).
\end{align}
Define $\Phi(G(z),\varphi)$ for $G(z)\in \h\ot \C((z))[[\hbar]]$ in the obvious way.
And extend the bilinear form $\<\cdot,\cdot\>$ on $\h$ to a $\C((z))[[\hbar]]$-linear form on $\h\ot \C((z))[[\hbar]]$.

\begin{de}
Let
$\mathcal H_Q$ be the set of linear maps
$\eta: \h\rightarrow \h\otimes \C((z))[[\hbar]]$ such that
\begin{align}\label{eq:eta-neg-cond}
    &\eta_0(\al_i,z):=\eta(\al_i,z)|_{\hbar=0}\in \h\ot z\C[[z]],\quad
    \eta(\al_i,z)^-\in \h\ot \hbar\C[z\inv][[\hbar]]\quad\te{for }i\in I.
\end{align}
\end{de}

From \cite{JKLT-Quantum-lattice-va} (cf. \cite{JKLT-Defom-va}), there exists a unique $B_Q$-comodule VA structure $\rho:V_Q[[\hbar]]\to V_Q[[\hbar]]\wh\ot B_Q$, such that
\begin{align*}
    \rho(h)=h\ot 1+1\ot h,\quad \rho(e_\al)=e_\al\ot e^\al\quad \te{for } h\in\h,\,\al\in Q.
\end{align*}
Fix an $\eta\in \mathcal H_Q$.
We also have a unique $B_Q$-module VA structure $\eta(\cdot,z):B_Q\to \E_\hbar(V_Q[[\hbar]])$,
such that
\begin{align*}
    \eta(h,z)=\Phi(\eta'(h,z),Y_Q),\quad \eta(e^\al,z)=\exp\( \Phi(\eta(\al,z),Y_Q) \)
\end{align*}
for $h\in\h$ and $\al\in Q$.
It is easy to check that $\tau$ is an invertible element in $\mathfrak L_{B_Q}^\rho(V_Q[[\hbar]])$.
So we get an $\hbar$-quantum VA $\mathfrak D_\eta^\rho(V_Q[[\hbar]])$.
We denote $\mathfrak D_\eta^\rho(V_Q[[\hbar]])$ by $V_Q^\eta$, and denote its vertex operator map by $Y_Q^\eta$.
The following result was presented in \cite[Theorem 4.5, Proposition 4.9]{JKLT-Quantum-lattice-va}

\begin{thm}\label{thm:qlatticeVA}
The vertex operator map $Y_Q^\eta$ is uniquely determined by
\begin{align*}
    &Y_Q^\eta(h,z)=Y_Q(h,z)+\Phi(\eta'(h,z),Y_Q)\quad \te{for }h\in\h,\\
    &Y_Q^\eta(e_\al,z)=Y_Q(e_\al,z)\exp\( \Phi(\eta(\al,z),Y_Q) \)\quad \te{for }\al\in Q.
\end{align*}
The quantum Yang-Baxter operator $S_Q^\eta(z)$ is determined by ($h,h'\in\h$, $\al,\beta\in Q$):
\begin{align*}
    &S_Q^\eta(z)(h'\ot h)=h'\ot h+\vac\ot\vac\ot (\<\eta''(h',z),h\>-\<\eta''(h,-z),h'\>),\\
    &S_Q^\eta(z)(e_\al\ot h)=e_\al\ot h+e_\al\ot \vac\ot (\<\eta'(\al,z),h\>+\<\eta'(h,-z),\al\>),\\
    &S_Q^\eta(z)(h\ot e_\al)=h\ot e_\al-\vac\ot e_\al\ot (\<\eta'(\al,-z),h\>+\<\eta'(h,z),\al\>),\\
    &S_Q^\eta(z)(e_\beta\ot e_\al)=e_\beta\ot e_\al\ot \exp\( \<\eta(\al,-z),\beta\>-\<\eta(\beta,z),\al\> \).
\end{align*}
Moreover, for $h,h'\in\h$ and $\al,\beta\in Q$, we have that
\begin{align*}
    &Y_Q^\eta(h,z)^-h'=\(\<h,h'\>z^{-2}-\<\eta''(h,z)^-,h'\>\)\vac,\\
    &Y_Q^\eta(h,z)^-e_\al=\( \<h,\al\>z\inv+\<\eta'(h,z)^-,\al\> \)e_\al,\\
    &Y_Q^\eta(e_\al,z)e_\beta=\epsilon(\al,\beta)z^{\<\al,\beta\>}e^{\<\eta(\al,z),\beta\>}E^+(\al,z)e_{\al+\beta}.
\end{align*}
\end{thm}

Combining Theorem \ref{thm:smash} and Theorem \ref{thm:deform-va-h}, we get that

\begin{prop}\label{prop:latticeVA-mod-to-qlatticeVA-mod}
There is a functor $\mathfrak D_\eta$ from the category of $V_Q[[\hbar]]$-modules to the category of $V_Q^\eta$-modules.
More precisely, let $(W,Y_W)$ be a $V_Q[[\hbar]]$-module. Then there exists a $V_Q^\eta$-module structure $Y_W^\eta$ on $W$,
which is uniquely determined by
\begin{align*}
    &Y_W^\eta(h,z)=Y_W(h,z)+\Phi(\eta'(h,z),Y_W)\quad \te{for }h\in\h,\\
    &Y_W^\eta(e_\al,z)=Y_W(e_\al,z)\exp\(\Phi(\eta(\al,z),Y_W)\)\quad \te{for }\al\in Q.
\end{align*}
We denote this module by $W^\eta$.
In addition, let $W_1$ be another $V_Q[[\hbar]]$-module, and let $f:W\to W_1$ be a $V_Q[[\hbar]]$-module homomorphism.
Then $f$ is also a $V_Q^\eta$-module homomorphism.
\end{prop}

In \cite[Section 6.5]{LL}, Lepowsky and H. Li proved that the category of $V_Q$-modules is isomorphic to certain module category of an associative algebra $A(Q)$. In the rest of this section, we give an $\hbar$-adic analogue of this module category.
By using this, we prove the completely reducibility of $V_Q^\eta$-modules, and determine all simple $V_Q^\eta$-modules.

\begin{de}
Define $\mathcal A_\hbar^\eta(Q)$ to be the category, where the objects are topologically free $\C[[\hbar]]$-modules $W$ equipped with fields $\al_{i,\hbar}(z),\,e_{i,\hbar}^\pm(z)\in\E_\hbar(W)$ ($i\in I$), satisfying the conditions that ($i,j\in I$):
\begin{align}
  \tag{AQ1}\label{AQ1}&[\al_{i,\hbar}(z_1),\al_{j,\hbar}(z_2)]=\<\al_i,\al_j\>\pd{z_2}z_1\inv\delta\(\frac{z_2}{z_1}\)\\
  &\quad\nonumber-\<\eta''(\al_i,z_1-z_2),\al_j\>+\<\eta''(\al_j,z_2-z_1),\al_i\>,\\
  \tag{AQ2}\label{AQ2}&[\al_{i,\hbar}(z_1),e_{j,\hbar}^\pm(z_2)]=\pm\<\al_i,\al_j\>e_{j,\hbar}^\pm(z_2)z_1\inv\delta\(\frac{z_2}{z_1}\)\\
  &\quad\pm\<\eta'(\al_i,z_1-z_2),\al_j\>
    e_{j,\hbar}^\pm(z_2)\pm\<\eta'(\al_j,z_2-z_1),\al_i\>e_{j,\hbar}^\pm(z_2),\nonumber\\
  \tag{AQ3}\label{AQ3}&\iota_{z_1,z_2}e^{-\<\eta(\al_i,z_1-z_2),\al_j\>}P_{ij}(z_1-z_2)
    e_{i,\hbar}^\pm(z_1)e_{j,\hbar}^\pm(z_2)\\
    &\quad=\iota_{z_2,z_1}e^{-\<\eta(\al_j,z_2-z_1),\al_i\>}P_{ij}(-z_2+z_1)e_{j,\hbar}^\pm(z_2)e_{i,\hbar}^\pm(z_1),\nonumber\\
    &\quad\quad\te{for some }P_{ij}(z)\in\C((z))[[\hbar]],\nonumber\\
  \tag{AQ4}\label{AQ4}&\iota_{z_1,z_2}e^{\<\eta(\al_i,z_1-z_2),\al_j\>}Q_{ij}(z_1-z_2)
    e_{i,\hbar}^\pm(z_1)e_{j,\hbar}^\mp(z_2)\\
    &\quad=\iota_{z_2,z_1}e^{\<\eta(\al_j,z_2-z_1),\al_i\>}Q_{ij}(-z_2+z_1)e_{j,\hbar}^\mp(z_2)e_{i,\hbar}^\pm(z_1),\nonumber\\
    &\quad\nonumber\quad\te{for some }Q_{ij}(z)\in\C((z))[[\hbar]],\\
  \tag{AQ5}\label{AQ5}&\frac{d}{dz}e_{i,\hbar}^\pm(z)
    =\pm \al_{i,\hbar}(z)^+e_{i,\hbar}^\pm(z)\pm e_{i,\hbar}^\pm(z)\al_{i,\hbar}(z)^-
    -\<\eta'(\al_i,0)^+,\al_i\>e_{i,\hbar}^\pm(z),\\
  \tag{AQ6}\label{AQ6}&\iota_{z_1,z_2}e^{\<\eta(\al_i,z_1-z_2),\al_i\>}(z_1-z_2)^{\<\al_i,\al_i\>}
    e_{i,\hbar}^+(z_1)e_{i,\hbar}^-(z_2)\\
    &\quad=\iota_{z_2,z_1}e^{\<\eta(\al_i,z_2-z_1),\al_i\>}(z_1-z_2)^{\<\al_i,\al_i\>}
    e_{i,\hbar}^-(z_2)e_{i,\hbar}^+(z_1),\nonumber\\
  &\te{and}\,\,\left.\(e^{\<\eta(\al_i,z_1-z_2),\al_i\>}(z_1-z_2)^{\<\al_i,\al_i\>}
    e_{i,\hbar}^+(z_1)e_{i,\hbar}^-(z_2)\)\right|_{z_1=z_2}=1.\nonumber
\end{align}
The morphisms between two objects $W_1$ and $W_2$
%\begin{align*}
%(W_1,\{\beta_{i,\hbar}(z)\}_{i\in I},\{e_{i,\hbar}^\pm(z)\}_{i\in I}),\quad
%(W_2,\{\beta_{i,\hbar}(z)\}_{i\in I},\{e_{i,\hbar}^\pm(z)\}_{i\in I})
%\end{align*}
are $\C[[\hbar]]$-module maps $f:W_1\to W_2$, and
\begin{align*}
    \al_{i,\hbar}(z)\circ f=f\circ \al_{i,\hbar}(z),\quad e_{i,\hbar}^\pm(z)\circ f=f\circ e_{i,\hbar}^\pm(z)\quad \te{for }i\in I.
\end{align*}
\end{de}

By applying Remark \ref{rem:Jacobi-S} to Theorem \ref{thm:qlatticeVA}, one can verify the following result straightforwardly.
\begin{prop}\label{prop:qlatticeVA-mod-to-A-mod}
There exists a functor $\mathfrak I$ from the category of $V_Q^\eta$-modules to the category $\mathcal A_\hbar^\eta(Q)$.
More precisely, let $(W,Y_W^\eta)$ be a $V_Q^\eta$-module. Then
\begin{align*}
    \(W,\{Y_W^\eta(\al_i,z)\}_{i\in I},\{Y_W^\eta(e_{\pm \al_i},z)\}_{i\in I}\)\in \obj \mathcal A_\hbar^\eta(Q),
\end{align*}
where
\begin{align*}
    P_{ij}(z)=z^{-\<\al_i,\al_j\>},\quad Q_{ij}(z)=z^{\<\al_i,\al_j\>}\quad \te{for }i,j\in I.
\end{align*}
Moreover, let $W_1$ be another $V_Q^\eta$-module, and let $f:W\to W_1$ be a $V_Q^\eta$-module.
Then $f$ is also a morphism in $\mathcal A_\hbar^\eta(Q)$.
\end{prop}

We will prove the following result in Section \ref{subsec:pf-undeform-qlattice}.

\begin{thm}\label{thm:Undeform-qlattice}
There exists a functor $\mathfrak U_\eta$ from the category $\mathcal A_\hbar^\eta(Q)$ to the category of $V_Q[[\hbar]]$-modules.
Moreover, $\mathfrak U_\eta\circ\mathfrak I\circ \mathfrak D_\eta$ is the identity functor of the category of $V_Q[[\hbar]]$-modules, and $\mathfrak D_\eta\circ \mathfrak U_\eta\circ \mathfrak I$ is the identity functor of the category of $V_Q^\eta$-modules.
Furthermore, the category of $V_Q[[\hbar]]$-modules, the category of $V_Q^\eta$-modules and the category $\mathcal A_\hbar^\eta(Q)$ are isomorphic.
\end{thm}

%
%\begin{lem}
%The category of $V_Q$-modules is equivalent to the category of $V_Q[[\hbar]]$-modules.
%\end{lem}
%
%\begin{proof}
%For each $V_Q$-module $W$, $W[[\hbar]]$ naturally carries a $V_Q[[\hbar]]$-module structure.
%Conversely, for each $V_Q[[\hbar]]$-module $W$, $W/\hbar W$ naturally carries a $V_Q$-module structure.
%Let $\mathfrak A$ be the functor from the category of $V_Q$-modules to the category of $V_Q[[\hbar]]$-modules,
%such that $\mathfrak A(W)=W[[\hbar]]$,
%and let $\mathfrak B$ be the functor from the category of $V_Q[[\hbar]]$-modules to the category of $V_Q$-modules,
%such that $\mathfrak B(W)=W/\hbar W$.
%It is easy to see that $\mathfrak B\circ \mathfrak A$ is isomorphic to the identity functor of the category of $V_Q$-modules.
%Finally, let $W$ be a $V_Q[[\hbar]]$-module.
%Then $W$ is naturally a $V_Q$-module.
%Since $W$ is completely reducible (see Theorem \ref{thm:lattice-VA}) and $\hbar W$ is a $V_Q$-submodule, there is a $V_Q$-submodule $W_0$ of $W$, such that $W_0\oplus \hbar W=W$.
%Notice that $W_0\cong W/\hbar W$ as $V_Q$-modules and $W_0[[\hbar]]=W$ as $V_Q[[\hbar]]$-modules.
%Then we have that
%\begin{align*}
%    \mathfrak A\circ\mathfrak B(W)=(W/\hbar W)[[\hbar]]\cong W_0[[\hbar]]=W.
%\end{align*}
%Therefore, $\mathfrak A\circ\mathfrak B$ is isomorphic to the identity functor of the category of $V_Q[[\hbar]]$-modules.
%\end{proof}

Combining Lemmas \ref{lem:simple}, \ref{lem:simple2} and \ref{lem:completely-reducible} with Theorems \ref{thm:Undeform-qlattice} and \ref{thm:lattice-VA}, we have that
\begin{thm}
Every $V_Q^\eta$-module is completely reducible, and for any simple $V_Q^\eta$-module $W$, there exists $Q+\gamma\in Q^0/Q$, such that $W\cong V_{Q+\gamma}^\eta$.
\end{thm}

\begin{coro}\label{coro:qlattice-universal-property}
Let $(V,Y,\vac_V)$ be an $\hbar$-adic nonlocal VA containing a subset
$\set{\al_i,e_i^\pm}{i\in I}$, such that
the topologically free $\C[[\hbar]]$-module $V$ equipped with the fields
\begin{align*}
  Y(\al_i,z),\quad Y(e_i^\pm,z)\quad \te{for }i\in I
\end{align*}
becomes an object in $\mathcal A_\hbar^\eta(Q)$.
Then there is a unique $\hbar$-adic nonlocal VA homomorphism
$\psi:V_Q^\eta\to V$, such that
\begin{align*}
  \psi(\al_i)=\al_i,\quad \psi(e_{\pm\al_i})=e_i^\pm\quad\te{for }i\in I.
\end{align*}
Moreover, $\psi$ is injective.
\end{coro}

\begin{proof}
Theorem \ref{thm:Undeform-qlattice} provides a $V_Q^\eta$-module structure $Y_V^\eta(\cdot,z)$ on $V$, uniquely determined by
\begin{align}\label{eq:prop-qlattice-universal-temp0}
  &Y_V^\eta(\al_i,z)=Y(\al_i,z),\quad
   Y_V^\eta(e_{\pm\al_i},z)=Y(e_i^\pm,z)\quad\te{for }i\in I.
\end{align}
Let
\begin{align*}
  U=\Span_\C\set{\al_i,\,e_{\pm\al_i},\,\vac}{i\in I}.
\end{align*}
Define $\psi^0:U\to V$ by $\psi^0(\vac)=\vac_V$ and
\begin{align*}
  \psi^0(\al_i)=\al_i,\quad \psi^0(e_{\pm\al_i})=e_i^\pm\quad \te{for }i\in I.
\end{align*}
Then
\begin{align}\label{eq:prop-qlattice-universal-temp1}
  Y_V^\eta(u,z)\vac_V=Y(\psi^0(u),z)\vac_V\in V[[z]]\quad\te{for }u\in U.
\end{align}
From Theorem \ref{thm:qlatticeVA}, we see that
\begin{align*}
  S_Q^\eta(z)(U\ot U)\subset U\ot U\ot \C((z))[[\hbar]].
\end{align*}
Then we get from Proposition \ref{prop:vacuum-like} that
\begin{align}\label{eq:prop-qlattice-universal-temp2}
  Y_V^\eta(u,z)\vac_V\in V[[z]]\quad\te{for }u\in V_Q^\eta.
\end{align}
Notice that $V_Q^\eta$ is generated by $\set{\al_i,e_{\pm\al_i}}{i\in I}$.
Combining this fact with \eqref{eq:prop-qlattice-universal-temp0}, \eqref{eq:prop-qlattice-universal-temp2} and Proposition \ref{prop:vacuum-like-sp}, we get an $\hbar$-adic nonlocal VA homomorphism
\begin{align*}
  \psi:V_Q^\eta\to V,\quad\te{such that }\psi(u)=\psi^0(u)\quad\te{for }u\in U.
\end{align*}

Finally, we show that $\psi$ is injective.
View $V$ as a $V_Q^\eta$-module via $\psi$.
Then $\psi$ is also a $V_Q^\eta$-module homomorphism.
Since $V_Q^\eta$ itself is a simple $V_Q^\eta$-module and $\psi(\vac)=\vac_V\ne 0$, we get that $\psi$ must be injective.
\end{proof}

\begin{coro}\label{coro:h--to-e}
Let $(W,Y_W^\eta)$ be a $V_Q^\eta$-module.
Suppose that there exist $w\in W$ and $\lambda\in\h$ such that
\begin{align*}
  Y_W^\eta(\al_i,z)^-w=\<\lambda,\al_i\>wz\inv+\<\lambda,\eta'(\al_i,z)^-\>w\quad\te{for }i\in I.
\end{align*}
Then
\begin{align*}
  &e^{\mp\<\lambda,\eta(\al_i,z)\>}z^{\mp\<\lambda,\al_i\>}Y_W^\eta(e_{\pm\al_i},z)w\in W[[z]].
\end{align*}
\end{coro}

\begin{proof}
Theorem \ref{thm:Undeform-qlattice} provides a $V_Q[[\hbar]]$-module structure $Y_W$ on $W$, such that
\begin{align}
  &Y_W^\eta(\al_i,z)=Y_W(\al_i,z)+\Phi(\eta'(\al_i,z),Y_W),\label{eq:coro-h--to-e-temp1}\\
  &Y_W^\eta(e_{\pm\al_i},z)=Y_W(e_{\pm\al_i},z)\exp\(\pm\Phi(\eta(\al_i,z),Y_W)\),
\end{align}
for $i\in I$.
Write
\begin{align*}
  Y_W(\al_i,z)=\sum_{n\in\Z}\al_i(n)z^{-n-1},\quad \eta(\al_i,z)=\sum_{j\in I}\al_j\ot f_{ij}(z).
\end{align*}
Taking $\Res_zz^n$ on both hand sides of \eqref{eq:coro-h--to-e-temp1} for $n\in\N$, we get that
\begin{align*}
  \Res_zz^n Y_W^\eta(\al_i,z)w
  =&\Res_zz^nY_W(\al_i,z)w+\sum_{j\in I}\Res_{z,z_1}z^n\pd{z}f_{ij}(z-z_1)Y_W(\al_j,z_1)w \\
  =&\al_i(n)w+\sum_{k\ge 0}\sum_{j\in I}\binom{n}{k}\al_j(k)w\Res_zz^{n-k}\pd{z}f_{ij}(z)\\
  =&\al_i(n)w+\sum_{k=0}^{n-1}\binom{n}{k}\sum_{j\in I}\al_j(k)w\Res_zz^{n-k}\pd{z}f_{ij}(z).
\end{align*}
Viewing $\al_i(n)w$ ($i\in I$, $n\in\N$) as indeterminates, the linear system above has the unique solution:
\begin{align*}
  \al_i(n)w=\delta_{n,0}\<\lambda,\al_i\>w,\quad i\in I,\,n\in\N.
\end{align*}
Then
\begin{align*}
  &Y_W^\eta(e_{\pm\al_i},z)w=Y_W(e_{\pm\al_i},z)\exp\(\pm\Phi(\eta(\al_i,z),Y_W)\)w\\
  =&E^-(\mp\al_i,z)e_{\pm\al_i}E^+(\mp\al_i,z)z^{\pm\al_i}\exp\(\pm\Phi(\eta(\al_i,z),Y_W)\)w\\
  =&E^-(\mp\al_i,z)e_{\pm\al_i}w z^{\pm\<\lambda,\al_i\>}\exp\(\pm\<\lambda,\eta(\al_i,z)\>\).
\end{align*}
As $E^-(\mp\al_i,z)e_{\pm\al_i}w\in W[[z]]$, we complete the proof.
\end{proof}

\section{Proof of Theorem \ref{thm:Undeform-qlattice}}\label{subsec:pf-undeform-qlattice}
In this section, we fix an object
\begin{align*}
(W,\{\al_{i,\hbar}(z)\}_{i\in I},\{e_{i,\hbar}^\pm(z)\}_{i\in I})\in\obj\mathcal A_\hbar^\eta(Q).
\end{align*}

\begin{lem}\label{lem:linear-sys}
Let $V$ be a $\C[[\hbar]]$-module. For each linear map $\varphi(\cdot,z):\h\to V[[z^{\pm1}]]$, there exists a unique linear map $\bar\varphi(\cdot,z):\h\to V[[z^{\pm1}]]$, such that
\begin{align}\label{eq:linear-sys}
  \bar\varphi(\al_i,z)+\Phi(\eta'(\al_i,z),\bar\varphi)=\varphi(\al_i,z)\quad\te{for }i\in I.
\end{align}
Moreover, if $\varphi(\h,z)^-\subset V[z\inv][[\hbar]]$, then
\begin{align*}
  \bar\varphi(\h,z)^-\subset V[z\inv][[\hbar]].
\end{align*}
\end{lem}

\begin{proof}
Write
\begin{align*}
  \varphi(\al_i,z)=\sum_{n\in\Z}\al_i(n)z^{-n-1},\quad
  \bar\varphi(\al_i,z)=\sum_{n\in\Z}\bar\al_i(n)z^{-n-1},\quad \eta(\al_i,z)=\sum_{j\in I}\al_j\ot f_{ij}(z),
\end{align*}
and write $f_{ij}(z)=\sum_{n\in\Z}f_{ij}(n)z^{-n-1}$.
We view $\bar\al_i(n)$ ($i\in I$, $n\in\Z$) as indeterminates.
Taking $\Res_zz^n$ on both hand sides of \eqref{eq:linear-sys} for $n\in\N$, we get the following linear system
\begin{align*}
  &\al_i(n)=\Res_zz^n\varphi(\al_i,z)\\
  =&\bar\al_i(n)+\sum_{j\in I}\sum_{k\in\N}\Res_zz^n\bar\al_i(k)\frac{(-1)^k}{k!}\pdiff{z}{k+1}f_{ij}(z)\\
  =&\bar\al_i(n)+\sum_{j\in I}\sum_{k\in\N}\binom{n}{k}\bar\al_i(k)\Res_zz^{n-k}\pd{z}f_{ij}(z)\\
  =&\bar\al_i(n)-\sum_{j\in I}\sum_{k=0}^{n-1}\binom{n}{k}(n-k)f_{ij}(n-k-1)\bar\al_i(k).
\end{align*}
Since $I$ is a finite set, we write $I=\{1,2,\dots,m\}$.
Introduce a total order on the set $I\times\N$ as follows:
\begin{align*}
  (i, n_1)<(j, n_2)\Longleftrightarrow n_1<n_2\,\,\te{or}\,\,(n_1=n_2\,\,\te{and}\,\,i<j).
\end{align*}
Note that there exists a bijective map from $I\times \N$ to $\N$ that preserves the total orders.
So we can rearrange the set $\{\al_i(n)\}_{i\in I,n\in\N}$ (resp. $\{\bar\al_i(n)\}_{i\in I,n\in\N}$) as follows
$\{\al_{[n]}\}_{n\in\N}$ (resp. $\{\bar\al_{[n]}\}_{n\in\N}$).
Then the linear system given above can be rewritten as follows
\begin{align}\label{eq:lin-sys}
  C(\bar\al_{[1]},\bar\al_{[2]},\dots)^t=(\al_{[1]},\al_{[2]},\dots)^t,
\end{align}
where $C=(c_{ij})_{i,j\in\N}$ is a lower unitriangular matrix.
From the assumption of $\eta$ (see \eqref{eq:eta-neg-cond}), we see that $C$ satisfies the following conditions:

(C1) $c_{ij}\in\hbar\C[[\hbar]]$ for any $i>j$;
(C2) for each $N\in\Z_+$, there exists $k\in\Z_+$, such that for any $j\in\N$, one has that $c_{ij}\in\hbar^N\C[[\hbar]]$ for all $i>j+k$.

Then $C$ is invertible and its inverse $C\inv$ is also a lower unitriangular matrix satisfying the above two conditions (C1) and (C2).
It means that the linear system \eqref{eq:lin-sys} has a unique solution.
Hence, $\bar\al_i(n)$ ($i\in I$, $n\in\N$) are uniquely determined.
Taking $\Res_zz^{-n-1}$ on both hand sides of \eqref{eq:linear-sys} for $n\in\N$, we get that
\begin{align*}
  &\al_i(-n-1)=\Res_zz^{-n-1}\varphi(\al_i,z)\\
  =&\bar\al_i(-n-1)+\sum_{j\in I}\sum_{k\in\N}\Res_zz^{-n-1}\bar\al_i(k)\frac{(-1)^k}{k!}\pdiff{z}{k+1}f_{ij}(z)\\
  =&\bar\al_i(-n-1)-\sum_{j\in I}\sum_{k\in\N}(-1)^k\binom{n+k}{k}(n+k+1)f_{ij}(-n-k-3)\bar \al_i(k).
\end{align*}
So $\bar\al_i(-n-1)$ ($i\in I$, $n\in\N$) are also uniquely determined.
The moreover statements follows from the fact that $C\inv$ satisfies the conditions (C1) and (C2).
\end{proof}

Under the consideration of Lemma \ref{lem:linear-sys}, we have the following definition.

\begin{de}
Let $\varphi(\cdot,z):\h\to \E_\hbar(W)$ be the linear map determined by
\begin{align}\label{eq:def-beta-i}
  \varphi(\al_i,z)+\Phi(\eta'(\al_i,z),\varphi)=\al_{i,\hbar}(z)\quad\te{for }i\in I.
\end{align}
\end{de}

\begin{lem}
For $i,j\in I$, we have that
\begin{align}
  &[\varphi(\al_i,z_1),\al_{j,\hbar}(z_2)]=\<\al_i,\al_j\>\pd{z_2}z_1\inv\delta\(\frac{z_2}{z_1}\)
    +\<\eta''(\al_j,z_2-z_1),\al_i\>,\label{eq:bar-h-h-com}\\
  &[\varphi(\al_i,z_1),e_{j,\hbar}^\pm(z_2)]=\pm\<\al_i,\al_j\>e_{j,\hbar}^\pm(z_2)z_1\inv\delta\(\frac{z_2}{z_1}\)
  \label{eq:bar-h-e-com}
  \pm\<\eta'(\al_j,z_2-z_1),\beta_i\>e_{j,\hbar}^\pm(z_2),\\
  &[\varphi(\al_i,z_1),\varphi(\al_j,z_2)]=\<\al_i,\al_j\>\pd{z_2}z_1\inv\delta\(\frac{z_2}{z_1}\).\label{eq:AQ-1}
\end{align}
\end{lem}

\begin{proof}
Define
\begin{align*}
  \varphi_1(\cdot,z):\h&\longrightarrow \End_{\C[[\hbar]]}(W)[[z_2^{\pm 1}]][[z,z\inv]];\quad
  \al_i\mapsto [\varphi(\al_i,z),\al_{j,\hbar}(z_2)].
\end{align*}
From \eqref{AQ1}, we get the following linear system
\begin{align*}
  &\varphi_1(\al_i,z)+\Phi(\eta'(\al_i,z),\varphi_1)
  =[\varphi(\al_i,z),\al_{j,\hbar}(z_2)]+[\Phi(\eta'(\al_i,z),\varphi),\al_{j,\hbar}(z_2)]\\
  =&[\al_{i,\hbar}(z),\al_{j,\hbar}(z_2)]
  =\<\al_i,\al_j\>\pd{z_2}z\inv\delta\(\frac{z_2}{z}\)
  -\<\eta''(\al_i,z-z_2),\al_j\>
  +\<\eta''(\al_j,z_2-z),\al_i\>.
\end{align*}
Lemma \ref{lem:linear-sys} provides the following unique solution:
\begin{align*}
  \varphi_1(\al_i,z)=\<\al_i,\al_j\>\pd{z_2}z\inv\delta\(\frac{z_2}{z}\)
    +\<\eta''(\al_j,z_2-z),\al_i\>,
\end{align*}
as desired.

Define
\begin{align*}
  \varphi_2(\cdot,z):\h&\longrightarrow \End_{\C[[\hbar]]}(W)[[z_2^{\pm 1}]][[z,z\inv]];\quad
  \al_i\mapsto [\varphi(\al_i,z),e_{j,\hbar}^\pm(z_2)].
\end{align*}
From (AQ2), we get the following linear system
\begin{align*}
  &\varphi_2(\al_i,z)+\Phi(\eta'(\al_i,z),\varphi_2)\\
  =&[\varphi(\al_i,z),e_{j,\hbar}^\pm(z_2)]+[\Phi(\eta'(\al_i,z),\varphi),e_{j,\hbar}^\pm(z_2)]
  =[\al_{i,\hbar}(z),e_{j,\hbar}^\pm(z_2)]\\
  =&\pm\<\al_i,\al_j\>e_{j,\hbar}^\pm(z_2)z_1\inv\delta\(\frac{z_2}{z}\)
  \pm \<\eta'(\al_i,z-z_2),\al_j\>e_{j,\hbar}^\pm(z_2)
   \pm \<\eta'(\al_j,z_2-z),\al_i\>e_{j,\hbar}^\pm(z_2).
\end{align*}
Lemma \ref{lem:linear-sys} provides the following unique solution:
\begin{align*}
  \varphi_2(\al_i,z)=\pm\<\al_i,\al_j\>e_{j,\hbar}^\pm(z_2)z\inv\delta\(\frac{z_2}{z}\)
    \pm\<\eta'(\al_j,z_2-z),\al_i\>e_{j,\hbar}^\pm(z_2),
\end{align*}
as desired.

Finally, we define
\begin{align*}
  \varphi_3(\cdot,z):\h&\longrightarrow \End_{\C[[\hbar]]}(W)[[z_2^{\pm 1}]][[z,z\inv]];\quad
  \al_i\mapsto [\varphi(\al_i,z),\varphi(\al_j,z_2)].
\end{align*}
Then we get from \eqref{eq:bar-h-h-com} that
\begin{align*}
  &\varphi_3(\al_i,z)+\Phi(\eta'(\al_i,z),\varphi_3)
  =[\varphi(\al_i,z),\varphi(\al_j,z_2)]+[\Phi(\eta'(\al_i,z),\varphi),\varphi(\al_j,z_2)]\\
  =&
  [\al_{i,\hbar}(z),\varphi(\al_j,z_2)]
  =\<\al_i,\al_j\>\pd{z_2}z\inv\delta\(\frac{z_2}{z}\)
  -\<\eta''(\al_i,z-z_2),\al_j\>.
\end{align*}
Lemma \ref{lem:linear-sys} provides the following unique solution:
\begin{align*}
  \varphi_3(\al_i,z)=\<\al_i,\al_j\>\pd{z_2}z\inv\delta\(\frac{z_2}{z}\),
\end{align*}
as desired.
\end{proof}

\begin{prop}\label{prop:A-Q-rels}
For $i\in I$, we define
\begin{align}
  &e_i^\pm(z)=e_{i,\hbar}^\pm(z)\exp\(\mp\Phi(\eta(\al_i,z),\varphi)\).\label{eq:def-e-i}
\end{align}
Then we have that
\begin{align}
  &[\varphi(\al_i,z_1),e_j^\pm(z_2)]
  =\pm\<\al_i,\al_j\>e_j^\pm(z_2)z_1\inv\delta\(\frac{z_2}{z_1}\),\label{eq:AQ-2}\\
  & \iota_{z_1,z_2}P_{ij}(z_1-z_2)e_i^\pm(z_1)e_j^\pm(z_2)
                =\iota_{z_2,z_1}P_{ij}(z_1-z_2)e_j^\pm(z_2)e_i^\pm(z_1),\label{eq:AQ-3}\\
  & \iota_{z_1,z_2}Q_{ij}(z_1-z_2)e_i^\pm(z_1)e_j^\mp(z_2)
                =\iota_{z_2,z_1}Q_{ij}(z_1-z_2)e_j^\mp(z_2)e_i^\pm(z_1),\label{eq:AQ-4}\\
  & \frac{d}{dz}e_i^\pm(z)=\pm\varphi(\al_i,z)^+e_i^\pm(z)\pm e_i^\pm(z)\varphi(\beta_i,z)^-,\label{eq:AQ-5}\\
  & \iota_{z_1,z_2}(z_1-z_2)^{\<\al_i,\al_i\>}e_i^+(z_1)e_j^-(z_2)\label{eq:AQ-6}
                =\iota_{z_2,z_1}(z_1-z_2)^{\<\al_i,\al_i\>}e_j^-(z_2)e_i^+(z_1),\\
  & \((z_1-z_2)^{\<\al_i,\al_i\>}e_i^+(z_1)e_i^-(z_2)\)|_{z_1=z_2}=1,\label{eq:AQ-7}
\end{align}
where $P_{ij}(z)$ is given in \eqref{AQ3} and $Q_{ij}(z)$ is given in \eqref{AQ4}.
\end{prop}

\begin{proof}
From \eqref{eq:AQ-1}, we have that
\begin{align*}
  &[\Phi(\al_i\ot g(z_1),\varphi),e_{j,\hbar}^\pm(z_2)]
  =\pm\<\al_i,\al_j\>e_{j,\hbar}^\pm(z_2)g(z_1-z_2)
\end{align*}
for any $g(z)\in\C((z))[[\hbar]]$.
By linearity, we get that
\begin{align}
  &[\Phi(\eta'(\al_i,z_1),\varphi),e_{j,\hbar}^\pm(z_2)]
  =\pm \<\eta'(\al_i,z_1-z_2),\al_j\>e_{j,\hbar}^\pm(z_2),\label{eq:Phi-e-h-com-der}\\
  &[\Phi(\eta(\al_i,z_1),\varphi),e_{j,\hbar}^\pm(z_2)]
  =\pm \<\eta(\al_i,z_1-z_2),\al_j\>e_{j,\hbar}^\pm(z_2).\label{eq:Phi-e-h-com}
\end{align}
%Then we have that
%\begin{align*}
%  &[\beta_i(z_1),\beta_j(z_2)]\\
%  =&[(\beta_i)_\hbar(z_1)-\Phi(\eta'(\beta_i,z_1),\bar\varphi),
%  (\beta_j)_\hbar(z_2)-\Phi(\eta'(\beta_j,z_2),\bar\varphi)]\\
%  =&\<\beta_i,\beta_j\>\pd{z_2}z_1\inv\delta\(\frac{z_2}{z_1}\)\\
%  &\quad -\<\eta''(\beta_i,z_1-z_2),\beta_j\>
%    +\<\eta''(\beta_j,z_2-z_1),\beta_i\>\\
%  &\quad+\<\eta''(\beta_i,z_1-z_2),\beta_j\>
%    -\<\eta''(\beta_j,z_2-z_1),\beta_i\>\\
%  =&\<\beta_i,\beta_j\>\pd{z_2}z_1\inv\delta\(\frac{z_2}{z_1}\).
%\end{align*}
%We complete the proof of \eqref{eq:AQ-1}.
%
%From \eqref{eq:bar-h-e-com}, we get that
%\begin{align*}
%    [\bar\beta_i(z_1)^-,e_{j,\hbar}^\pm(z_2)]=\pm\<\beta_i,\beta_j\>e_{j,\hbar}^\pm(z_2)(z_1-z_2)\inv.
%\end{align*}
%Then for any $g(z)\in\C((z))[[\hbar]]$, we have that
%\begin{align*}
%    [\Phi(\beta_i\ot g(z_1),\bar\varphi),e_{j,\hbar}^\pm(z_2)]=\pm \<\beta_i,\beta_j\>e_{j,\hbar}^\pm(z_2)g(z_1-z_2).
%\end{align*}
%By linearity, we get that
%\begin{align}
%    &[\Phi(\eta(\beta_i,z_1),\bar\varphi),e_{j,\hbar}^\pm(z_2)]=\pm\<\eta(\beta_i,z_1-z_2),\beta_j\>e_{j,\hbar}^\pm(z_2),\label{eq:Phi-e-h-com}\\
%    &[\Phi(\eta'(\beta_i,z_1),\bar\varphi),e_{j,\hbar}^\pm(z_2)]=\pm\<\eta'(\beta_i,z_1-z_2),\beta_j\>e_{j,\hbar}^\pm(z_2).
%\end{align}
Then we have that
\begin{align*}
  &[\varphi(\al_i,z_1),e_j^\pm(z_2)]
  =[\varphi(\al_i,z_1),e_{j,\hbar}^\pm(z_2)\exp\(\mp\Phi(\eta(\al_i,z_2),\varphi)\)]\\
  =&[\varphi(\al_i,z_1),e_{j,\hbar}^\pm(z_2)]\exp\(\mp\Phi(\eta(\al_i,z_2))\)
  +e_{j,\hbar}^\pm(z_2)[\varphi(\al_i,z_1),\exp\(\mp\Phi(\eta(\al_i,z_2))\)]\\
  =&\pm \<\al_i,\al_j\>e_{j,\hbar}^\pm(z_2)z_1\delta\(\frac{z_2}{z_1}\)\pm \<\eta'(\al_j,z_2-z_1),\al_i\>e_{j,\hbar}^\pm(z_2)
  \mp \<\eta'(\al_j,z_2-z_1),\al_i\>e_{j,\hbar}^\pm(z_2)\\
  =&\pm \<\al_i,\al_j\> e_{j,\hbar}^\pm(z_2)z_1\delta\(\frac{z_2}{z_1}\),
\end{align*}
where the second equation follows from \eqref{eq:bar-h-e-com} and \eqref{eq:Phi-e-h-com}.
It proves \eqref{eq:AQ-2}.

From \eqref{eq:Phi-e-h-com}, we have that
\begin{align*}
  &e_i^{\epsilon_1}(z_1)e_j^{\epsilon_2}(z_2)
  =e_{i,\hbar}^{\epsilon_1}(z_1)\exp\(-{\epsilon_1}\Phi(\eta(\al_i,z_1),\varphi)\)
  e_{j,\hbar}^{\epsilon_2}(z_2)\exp\(-{\epsilon_2}\Phi(\eta(\al_j,z_2),\varphi)\)\\
  =&e^{-{\epsilon_1}\epsilon_2\<\eta(\al_i,z_1-z_2),\al_j\>}
  e_{i,\hbar}^{\epsilon_1}(z_1)
  e_{j,\hbar}^{\epsilon_2}(z_2)
  \exp\(-{\epsilon_1}\Phi(\eta(\al_i,z_1),\varphi)\)
    \exp\(-{\epsilon_2}\Phi(\eta(\al_j,z_2),\varphi)\).
\end{align*}
Then we get from (AQ3) that
\begin{align*}
  &\iota_{z_1,z_2}P_{ij}(z_1-z_2)e_i^\pm(z_1)e_j^\pm(z_2)
  -\iota_{z_2,z_1}P_{ij}(z_1-z_2)e_j^\pm(z_2)e_i^\pm(z_1)\\
  =&\iota_{z_1,z_2}P_{ij}(z_1-z_2)e^{-\<\eta(\al_i,z_1-z_2),\al_j\>}e_{i,\hbar}^\pm(z_1)e_{j,\hbar}^\pm(z_2)\\
  &\quad\times\exp\(\mp\Phi(\eta(\al_i,z_1),\varphi)\)\exp\(\mp\Phi(\eta(\al_j,z_2),\varphi)\)\\
  &-\iota_{z_2,z_1}P_{ij}(z_1-z_2)e^{-\<\eta(\al_j,z_2-z_1),\al_i\>}
  e_{j,\hbar}^\pm(z_2)e_{i,\hbar}^\pm(z_1)\\
  &\quad\times\exp\(\mp\Phi(\eta(\al_i,z_1),\varphi)\)\exp\(\mp\Phi(\eta(\al_j,z_2),\varphi)\)\\
  =&0.
\end{align*}
And we get from (AQ4) that
\begin{align*}
  &\iota_{z_1,z_2}Q_{ij}(z_1-z_2)e_i^\pm(z_1)e_j^\mp(z_2)
  -\iota_{z_2,z_1}Q_{ij}(z_1-z_2)e_j^\mp(z_2)e_i^\pm(z_1)\\
  =&\iota_{z_1,z_2}Q_{ij}(z_1-z_2)e^{\<\eta(\al_i,z_1-z_2),\al_j\>}
  e_{i,\hbar}^\pm(z_1)e_{j,\hbar}^\mp(z_2)\\
  &\quad\times\exp\(\mp\Phi(\eta(\al_i,z_1),\varphi)\)\exp\(\pm\Phi(\eta(\al_j,z_2),\varphi)\)\\
  &-\iota_{z_2,z_1}Q_{ij}(z_1-z_2)e^{\<\eta(\al_j,z_2-z_1),\al_i\>}e_{j,\hbar}^\mp(z_2)e_{i,\hbar}^\pm(z_1)\\
  &\quad\times\exp\(\mp\Phi(\eta(\al_i,z_1),\varphi)\)\exp\(\pm\Phi(\eta(\al_j,z_2),\varphi)\)\\
  =&0.
\end{align*}
It proves \eqref{eq:AQ-3} and \eqref{eq:AQ-4}.
Similarly, we derive \eqref{eq:AQ-6} from the first equation of \eqref{AQ6}.
From \eqref{AQ5}, we get that
\begin{align*}
  &\frac{d}{dz}e_i^\pm(z)=\frac{d}{dz}e_{i,\hbar}^\pm(z)\exp\(\mp\Phi(\eta(\al_i,z),\varphi)\)\\
  =&\pm \beta_{i,\hbar}(z)^+e_{i,\hbar}^\pm(z)\exp\(\mp\Phi(\eta(\al_i,z),\varphi)\)
  +e_{i,\hbar}^\pm(z)\al_{i,\hbar}(z)^-\exp\(\mp\Phi(\eta(\al_i,z),\varphi)\)\\
  &-\<\eta'(\al_i,0)^+,\al_i\>e_{i,\hbar}^\pm(z)\exp\(\mp\Phi(\eta(\al_i,z),\varphi)\)\\
  &\mp e_{i,\hbar}^\pm(z)\Phi(\eta'(\al_i,z),\varphi)\exp\(\mp\Phi(\eta(\al_i,z),\varphi)\)\\
  =&\pm\(\varphi(\al_i,z)^++\Phi(\eta'(\al_i,z)^+,\varphi)\)e_{i,\hbar}^\pm(z)\exp\(\mp\Phi(\eta(\al_i,z),\varphi)\)\\
  &\pm e_{i,\hbar}^\pm(z)\(\varphi(\al_i,z)^-+\Phi(\eta'(\al_i,z)^-,\varphi)\)\exp\(\mp\Phi(\eta(\al_i,z),\varphi)\)\\
  &-\<\eta'(\al_i,0)^+,\al_i\>e_{i,\hbar}^\pm(z)\exp\(\mp\Phi(\eta(\al_i,z),\varphi)\)\\
  &\mp e_{i,\hbar}^\pm(z)\Phi(\eta'(\al_i,z),\varphi)\exp\(\mp\Phi(\eta(\al_i,z),\varphi)\)\\
  =&\pm\varphi(\al_i,z)^+e_i^\pm(z)\pm e_i^\pm(z)\varphi(\al_i,z)^-
  \pm[\Phi(\eta'(\al_i,z)^+,\varphi),e_{i,\hbar}^\pm(z)]\exp\(\mp\Phi(\eta(\al_i,z),\varphi)\)\\
  &-\<\eta'(\al_i,0)^+,\al_i\>e_i^\pm(z)\\
  =&\pm\varphi(\al_i,z)^+e_i^\pm(z)\pm e_i^\pm(z)\varphi(\al_i,z)^-,
\end{align*}
where the second equation follows from \eqref{eq:def-beta-i}, and the last equation follows from \eqref{eq:Phi-e-h-com-der}.
Then we complete the proof of \eqref{eq:AQ-5}.
Notice that
\begin{align*}
  &(z_1-z_2)^{\<\al_i,\al_i\>}e_i^+(z_1)e_i^-(z_2)\\
  =&e^{\<\eta(\al_i,z_1-z_2),\al_i\>}(z_1-z_2)^{\<\al_i,\al_i\>}e_{i,\hbar}^+(z_1)e_{i,\hbar}^-(z_2)
    \exp\(-\Phi(\eta(\al_i,z_1),\varphi)\)\exp\(\Phi(\eta(\al_i,z_2),\varphi)\).
\end{align*}
Then we get from \eqref{AQ6} that
\begin{align*}
  \((z_1-z_2)^{\<\al_i,\al_i\>}e_i^+(z_1)e_i^-(z_2)\)|_{z_1=z_2}=1.
\end{align*}
We complete the proof of \eqref{eq:AQ-7}.
\end{proof}

Similar to \eqref{eq:def-E}, we let ($i\in I$, $\epsilon=\pm$):
\begin{align*}
  \varphi(\al_i,z)=\sum_{n\in\Z}\al_i(n)z^{-n-1},\quad
  E^\pm(\epsilon\al_i,z)=\exp\(\epsilon\sum_{n\in\Z_+}\frac{\al_i(\pm n)}{\pm n}z^{\mp n}\).
\end{align*}
Inspired by the theory of $Z$-operators, we have the following result.

\begin{prop}\label{prop:Z-ops}
For each $i\in I$, we define
\begin{align}
  Z_i^\pm(z)=E^-(\pm \al_i,z)e_i^\pm(z)E^+(\pm\al_i,z)z^{\mp\al_i(0)}.
\end{align}
Then we have that
\begin{align}\label{eq:Z-noz}
  Z_i^\pm(z)\in\End_{\C[[\hbar]]}W,
\end{align}
that is, $Z_i^\pm(z)$ does not depend on $z$, so that we can write $Z_i^\pm$ for $Z_i^\pm(z)$.
Moreover,
\begin{align}
  &[\varphi(\al_i,z),Z_j^\pm]=\pm\<\al_i,\al_j\>Z_j^\pm z\inv,&&
  Z_i^\pm Z_j^\pm=(-1)^{\<\al_i,\al_j\>}Z_j^\pm Z_i^\pm,\label{eq:Z-ops-rel-1-2}\\
  &Z_i^\pm Z_j^\mp=(-1)^{\<\al_i,\al_j\>}Z_j^\mp Z_i^\pm,&& Z_i^+Z_i^-=1.\label{eq:Z-ops-rel-3-4}
\end{align}
\end{prop}

\begin{proof}
Notice that
\begin{align*}
  &\frac{d}{dz}E^-(-\epsilon\al_i,z)=\epsilon E^-(-\epsilon\al_i,z)\varphi(\al_i,z)^+\\
  &\frac{d}{dz}E^+(-\epsilon\al_i,z)z^{\epsilon\al_i(0)}
  =\epsilon\varphi(\al_i,z)^-E^+(-\epsilon\al_i,z)z^{\epsilon\al_i(0)}.
\end{align*}
Combining these relations with \eqref{eq:AQ-5}, we get that
\begin{align*}
  \frac{d}{dz}&Z_i^\pm(z)=\(\frac{d}{dz}E^-(\pm \al_i,z)\)e_i^\pm(z)E^+(\pm\al_i,z)z^{\mp\al_i(0)}\\
  &+E^-(\pm \al_i,z)\(\frac{d}{dz}e_i^\pm(z)\)E^+(\pm\al_i,z)z^{\mp\al_i(0)}\\
  &+E^-(\pm \al_i,z)e_i^\pm(z)\(\frac{d}{dz}E^+(\pm\al_i,z)z^{\mp\al_i(0)}\)\\
  =&\mp E^-(\pm \al_i,z)\varphi(\al_i,z)^+e_i^\pm(z)E^+(\pm\al_i,z)z^{\mp\al_i(0)}\\
  &\pm E^-(\pm \al_i,z)\varphi(\al_i,z)^+e_i^\pm(z)E^+(\pm\al_i,z)z^{\mp\al_i(0)}\\
  &\pm E^-(\pm \al_i,z)e_i^\pm(z)\varphi(\al_i,z)^-E^+(\pm\al_i,z)z^{\mp\al_i(0)}\\
  &\mp E^-(\pm \al_i,z)e_i^\pm(z)\varphi(\al_i,z)^-E^+(\pm\al_i,z)z^{\mp\al_i(0)}\\
  =&0.
\end{align*}
It shows that $Z_i^\pm(z)$ does not depend on $z$.

Notice that
\begin{align*}
  &E^+(-\epsilon\al_i,z_1)z_1^{\epsilon\al_i}e_j^\pm(z_2)
  =e_j^\pm(z_2)E^+(-\epsilon\al_i,z_1)z_1^{\epsilon\al_i}(z_1-z_2)^{\pm\epsilon\<\al_i,\al_j\>},\\
  &E^-(-\epsilon\al_i,z_1)e_j^\pm(z_2)
  =e_j^\pm(z_2)E^-(-\epsilon\al_i,z_1)\(1-z_2/z_1\)^{\mp\epsilon\<\al_i,\al_j\>},\\
  &E^+(-\epsilon_1\al_i,z_1)E^-(-\epsilon_2\al_j,z_2)
  =E^-(-\epsilon_2\al_j,z_2)E^+(-\epsilon_1\al_i,z_1)\(1-z_2/z_1\)^{\epsilon_1\epsilon_2\<\al_i,\al_j\>}.
\end{align*}
Then for $i,j\in I$, we have that
\begin{align}
  &\iota_{z_1,z_2}P_{ij}(z_1-z_2)(z_1-z_2)^{\<\al_i,\al_j\>}Z_i^\pm(z_1)Z_j^\pm(z_2)\label{eq:prop-Z-temp0}\\
  =&E^-(\pm \al_i,z_1)e_i^\pm(z_1)E^+(\pm\al_i,z_1)z_1^{\mp\al_i}
    E^-(\pm \al_j,z_2)
     e_j^\pm(z_2)E^+(\pm\al_j,z_2)z_2^{\mp\al_j}\nonumber\\
  =&E^-(\pm \al_i,z_1)E^-(\pm \al_j,z_2)
    \iota_{z_1,z_2}P_{ij}(z_1-z_2)e_i^\pm(z_1)e_j^\pm(z_2)\nonumber\\
   &\times E^+(\pm\al_i,z_1)z_1^{\mp\al_i}E^+(\pm\al_j,z_2)z_2^{\mp\al_j}.\nonumber
\end{align}
Similarly, we have that
\begin{align*}
  &\iota_{z_2,z_1}P_{ij}(z_1-z_2)(z_1-z_2)^{\<\al_i,\al_j\>}Z_j^\pm(z_2)Z_i^\pm(z_1)\\
  =&(-1)^{\<\al_i,\al_j\>}E^-(\pm \al_i,z_1)E^-(\pm \al_j,z_2)
    \iota_{z_2,z_1}P_{ij}(z_1-z_2)e_j^\pm(z_2)e_i^\pm(z_1)\\
  &\quad\times  E^+(\pm\al_i,z_1)z_1^{\mp\al_i}E^+(\pm\al_j,z_2)z_2^{\mp\al_j}.
\end{align*}
Then we get from \eqref{eq:AQ-3} that
\begin{align}\label{eq:prop-Z-temp1}
  &\iota_{z_1,z_2}P_{ij}(z_1-z_2)(z_1-z_2)^{\<\al_i,\al_j\>}Z_i^\pm(z_1)Z_j^\pm(z_2)\\
  =&(-1)^{\<\al_i,\al_j\>}\iota_{z_2,z_1}P_{ij}(z_1-z_2)(z_1-z_2)^{\<\al_i,\al_j\>}Z_j^\pm(z_2)Z_i^\pm(z_1).\nonumber
\end{align}
Let $0\ne g(z)\in\C[z][[\hbar]]$, such that
\begin{align*}
  g(z)P_{ij}(z)z^{\<\al_i,\al_j\>}\in\C[z][[\hbar]].
\end{align*}
By multiplying $g(z_1-z_2)$ on both hand sides of \eqref{eq:prop-Z-temp1}, we get that
\begin{align*}
  &g(z_1-z_2)P_{ij}(z_1-z_2)(z_1-z_2)^{\<\al_i,\al_j\>}Z_i^\pm(z_1)Z_j^\pm(z_2)\\
  =&(-1)^{\<\al_i,\al_j\>}g(z_1-z_2)P_{ij}(z_1-z_2)(z_1-z_2)^{\<\al_i,\al_j\>}Z_j^\pm(z_2)Z_i^\pm(z_1).
\end{align*}
It follows that
\begin{align*}
  Z_i^\pm Z_j^\pm=(-1)^{\<\al_i,\al_j\>}Z_j^\pm Z_i^\pm,
\end{align*}
since $Z_i^\pm(z)\in \E_{\C[[\hbar]]}(W)$ (see \eqref{eq:Z-noz}).
Similarly, we have that
\begin{align*}
  Z_i^\pm Z_j^\mp=(-1)^{\<\al_i,\al_j\>}Z_j^\mp Z_i^\pm.
\end{align*}

Finally, from a similar calculation as \eqref{eq:prop-Z-temp0}, we get that
\begin{align*}
  &Z_i^+(z_1)Z_i^-(z_2)
  =E^-(\al_i,z_1)E^-(- \al_i,z_2)\\
  \times
    \iota_{z_1,z_2}&(z_1-z_2)^{\<\al_i,\al_i\>}e_i^+(z_1)e_i^-(z_2)
    E^+(\al_i,z_1)z_1^{-\al_i}E^+(-\al_i,z_2)z_2^{\al_i}.
\end{align*}
Therefore,
\begin{align*}
  &Z_i^+Z_i^-=Z_i^+(z_1)Z_i^-(z_1)
  =\((z_1-z_2)^{\<\al_i,\al_i\>}e_i^+(z_1)e_i^-(z_2)\)
  |_{z_1=z_2}
  =1,
\end{align*}
where the last equation follows from \eqref{eq:AQ-7}.
\end{proof}

\begin{coro}\label{coro:group-alg-action}
There exists a unique $\C_\epsilon[Q]$-module action on $W$ such that
\begin{align*}
  e_{\pm\al_i}.w=Z_i^\pm w\quad\te{for }i\in I,\,\,w\in W.
\end{align*}
\end{coro}

\begin{prop}\label{prop:A-eta-mod-to-V-Q-mod}
There exists a unique $V_Q[[\hbar]]$-module structure $Y_W$ on $W$ such that
\begin{align*}
  Y_W(\al_i,z)=\varphi(\al_i,z),\quad Y_W(e_{\pm\al_i},z)=e_i^\pm(z)\quad\te{for }i\in I.
\end{align*}
\end{prop}

\begin{proof}
For $h=\sum_{i\in I}c_i\al_i\in\h$, we define
\begin{align*}
  h(z)=\sum_{i\in I}c_i\varphi(\al_i,z).
\end{align*}
In addition, if $\al\in Q$, we define
\begin{align*}
  e_\al(z)=E^-(-\al,z)E^+(-\al,z)e_\al z^{\al(0)},
\end{align*}
where the operator $e_\al\in \End(W)$ is defined as in Corollary \ref{coro:group-alg-action}.
One can straightforwardly verify that $e_0(z)=1$ and
\begin{align}
  &[h(z_1),h'(z_2)]=\<h,h'\>\pd{z_2}z_1\inv\delta\(\frac{z_2}{z_1}\),\\
  &[h(z_1),e_\beta(z_2)]=\<\al,\beta\>e_\beta(z_2)z_1\inv\delta\(\frac{z_2}{z_1}\),\\
  &(z_1-z_2)^{-\<\al,\beta\>}e_\al(z_1)e_\beta(z_2)=(-z_2+z_1)^{-\<\al,\beta\>}e_\beta(z_2)e_\al(z_1),\\
  &\frac{d}{dz}e_\al(z)=\al(z)^+e_\al(z)+e_\al(z)\al(z)^-,\\
  &(z_1-z_2)^{-\<\al,\beta\>-1}e_\al(z_1)e_\beta(z_2)
  -(-z_2+z_1)^{-\<\al,\beta\>-1}e_\beta(z_2)e_\al(z_1)\\
  &\quad=\epsilon(\al,\beta)e_{\al+\beta}(z_2)z_1\inv\delta\(\frac{z_2}{z_1}\),\nonumber
\end{align}
where $h,h'\in\h$ and $\al,\beta\in Q$.
It is proved in \cite[Section 6.5]{LL} that $W$ carries a $V_Q$-module structure such that
\begin{align*}
  Y_W(h,z)=h(z),\quad Y_W(e_\al,z)=e_\al(z)\quad\te{for }h\in\h,\,\al\in Q.
\end{align*}
The uniqueness follows from the fact that $V_Q$ is generated by $\h\cup\set{e_{\pm\al_i}}{i\in I}$.
\end{proof}

\noindent\emph{Proof of Theorem \ref{thm:Undeform-qlattice}:}
Proposition \ref{prop:A-eta-mod-to-V-Q-mod} provides a $V_Q[[\hbar]]$-module structure $Y_W(\cdot,z)$ on $W$.
Denote this module by $\mathfrak U_\eta(W)$.
Let $(W',\{\al_{i,\hbar}(z)\}_{i\in I}, \{e_{i,\hbar}^\pm(z)\}_{i\in I})$ be another object of $\mathcal A_\hbar^\eta(Q)$,
and let $f:W\to W'$ be a morphism.
That is, $f$ is a $\C[[\hbar]]$-module map and
\begin{align*}
    \al_{i,\hbar}(z)\circ f=f\circ \al_{i,\hbar}(z),\quad e_{i,\hbar}^\pm(z)\circ f=f\circ e_{i,\hbar}^\pm(z)\quad \te{for }i\in I.
\end{align*}
Define
\begin{align*}
  \varphi_1:\h&\longrightarrow \Hom_{\C[[\hbar]]}(W,W')[[z,z\inv]]\\
    \al_i&\mapsto \varphi(\al_i,z)\circ f-f\circ \varphi(\al_i,z).
\end{align*}
Then we get the following linear system
\begin{align*}
  &\varphi_1(\al_i,z)+\Phi(\eta'(\al_i,z),\varphi_1)\\
  =&\(\varphi(\al_i,z)+\Phi(\eta'(\al_i,z),\varphi)\)\circ f-f\circ \(\varphi(\al_i,z)+\Phi(\eta'(\al_i,z),\varphi)\)\\
  =&\al_{i,\hbar}(z)\circ f-f\circ \al_{i,\hbar}(z)=0\quad\te{for }i\in I,
\end{align*}
where the last equation follows from \eqref{eq:def-beta-i}.
Lemma \ref{lem:linear-sys} provides the following unique solution:
\begin{align*}
  \varphi(\al_i,z)\circ f-f\circ \varphi(\al_i,z)=\varphi_1(\al_i,z)=0\quad\te{for }i\in I
\end{align*}
From the definition of $e_i^\pm(z)$ (see \eqref{eq:def-e-i}), we also have that
\begin{align*}
  &e_i^\pm(z)\circ f=f\circ e_i^\pm(z)\quad\te{for }i\in I.
\end{align*}
Since $Y_W$ is uniquely determined by the following conditions (see Proposition \ref{prop:A-eta-mod-to-V-Q-mod})
\begin{align*}
  &Y_W(\al_i,z)=\varphi(\al_i,z),\quad Y_W(e_{\pm\al_i},z)=e_i^\pm(z)\quad \te{for }i\in I,
\end{align*}
$f:\mathfrak U_\eta(W)\to \mathfrak U_\eta(W')$ is a $V_Q[[\hbar]]$-module homomorphism.
We get a functor $\mathfrak U_\eta$ from the category $\mathcal A_\hbar^\eta(Q)$ to the category of $V_Q[[\hbar]]$-modules.

Next, we show that $\mathfrak U_\eta$ is the inverse of $\mathfrak I\circ\mathfrak D_\eta$.
Let $(W,Y_W)$ be a $V_Q[[\hbar]]$-module.
Proposition \ref{prop:latticeVA-mod-to-qlatticeVA-mod} provides a $V_Q^\eta$-module structure $Y_W^\eta$ on $W$ uniquely determined by
\begin{align}
    &Y_W^\eta(\al_i,z)=Y_W(\al_i,z)+\Phi(\eta'(\al_i,z),Y_W),\label{eq:Y-W-eta}\\
    &Y_W^\eta(e_{\pm\al_i},z)=Y_W(e_{\pm\al_i},z)\exp\(\pm \Phi(\eta(\al_i,z),Y_W) \)\quad\te{for }i\in I.\label{eq:Y-W-eta-2}
\end{align}
From Proposition \ref{prop:qlatticeVA-mod-to-A-mod}, we have that
\begin{align*}
    (W,\{Y_W^\eta(\al_i,z)\}_{i\in I},\{Y_W^\eta(e_{\pm\al_i},z)\}_{i\in I})\in \obj \mathcal A_\hbar^\eta(Q).
\end{align*}
Then by using Proposition \ref{prop:A-eta-mod-to-V-Q-mod}, we get another $V_Q[[\hbar]]$-module structure $Y_W'(\cdot,z)$ on $W$, such that
\begin{align}
    &Y_W'(\al_i,z)+\Phi(\eta'(\al_i,z),Y_W')=Y_W^\eta(\al_i,z),\label{eq:Y-W-beta-i}\\
    &Y_W'(e_{\pm\al_i},z)=Y_W^\eta(e_{\pm\al_i},z)\exp\(\mp \Phi(\eta(\al_i,z),Y_W') \)\quad\te{for }i\in I.\label{eq:Y-W-e-i}
\end{align}
In view of Lemma \ref{lem:linear-sys}, we get from \eqref{eq:Y-W-eta} and \eqref{eq:Y-W-beta-i} that
\begin{align*}
    Y_W(\al_i,z)=Y_W'(\al_i,z)\quad\te{for }i\in I.
\end{align*}
Consequently, we get from \eqref{eq:Y-W-eta-2} and \eqref{eq:Y-W-e-i} that
\begin{align*}
    &Y_W'(e_{\pm\al_i},z)=Y_W^\eta(e_{\pm\al_i},z)\exp\(\mp \Phi(\eta(\al_i,z),Y_W') \)\\
    =&Y_W^\eta(e_{\pm\al_i},z)\exp\(\mp \Phi(\eta(\al_i,z),Y_W \)
    =Y_W(e_{\pm\al_i},z)\qquad\te{for }i\in I.
\end{align*}
Therefore, $Y_W=Y_W'$, since $V_Q[[\hbar]]$ is generated by $\set{\al_i,e_{\pm\al_i}}{i\in I}$.
It proves that
$\mathfrak U_\eta\circ \mathfrak I\circ\mathfrak D_\eta$ is the identity functor of the category of $V_Q[[\hbar]]$-modules.

Let $(W,Y_W^\eta)$ be a $V_Q^\eta$-module.
Proposition \ref{prop:qlatticeVA-mod-to-A-mod} provides that
\begin{align*}
    (W,\{Y_W^\eta(\al_i,z)\}_{i\in I},\{Y_W^\eta(e_{\pm\al_i},z)\}_{i\in I})\in \obj \mathcal A_\hbar^\eta(Q).
\end{align*}
Using Proposition \ref{prop:A-eta-mod-to-V-Q-mod}, we get a $V_Q[[\hbar]]$-module structure $Y_W(\cdot,z)$ on $W$ uniquely determined by
\begin{align*}
    &Y_W(\al_i,z)+\Phi(\eta'(\al_i,z),Y_W)=Y_W^\eta(\al_i,z),\\
    &Y_W(e_{\pm\al_i},z)=Y_W^\eta(e_{\pm\al_i},z)\exp\(\mp \Phi(\eta(\al_i,z),Y_W) \).
\end{align*}
Then Proposition \ref{prop:latticeVA-mod-to-qlatticeVA-mod} provides another $V_Q^\eta$-module structure $\wh Y_W^\eta(\cdot,z)$ on $W$, such that
\begin{align*}
    &\wh Y_W^\eta(\al_i,z)=Y_W(\al_i,z)+\Phi(\eta'(\al_i,z),Y_W)=Y_W^\eta(\al_i,z),\\
    &\wh Y_W^\eta(e_{\pm\al_i},z)=Y_W(e_{\pm\al_i},z)\exp\( \pm\Phi(\eta(\al_i,z),Y_W) \)=Y_W^\eta(e_{\pm\al_i},z).
\end{align*}
Since $V_Q^\eta$ is generated by $\set{\al_i,e_{\pm\al_i}}{i\in I}$,
we get that $\wh Y_W^\eta=Y_W^\eta$.
Therefore, $\mathfrak D_\eta\circ\mathfrak U_\eta\circ \mathfrak I$ is the identity functor of the category of $V_Q^\eta$-modules.

\bibliographystyle{unsrt}
%\bibliographystyle{alpha}
%\bibliographystyle{alpha}

%\bibliography{../reference}


\end{document}